\documentclass[11pt,reqno]{amsart}
 \usepackage{xcolor}
 \usepackage{amsgen, amstext,amsbsy,amsopn, amsthm, amsfonts,amssymb,amscd,amsmath,euscript,enumerate,url,verbatim,calc,xypic,tikz}
 \usepackage{amsfonts}
\usepackage{amssymb} 
\usepackage{graphicx}
\usepackage{mathrsfs}  
\usepackage{hyperref}
\usepackage{MnSymbol}
\usepackage{pgfkeys}
\usepackage{mathtools}
\def\multiset#1#2{\ensuremath{\left(\kern-.3em\left(\genfrac{}{}{0pt}{}{#1}{#2}\right)\kern-.3em\right)}}

\DeclarePairedDelimiter\floor{\lfloor}{\rfloor}
\usetikzlibrary{arrows}
\newcommand{\midarrow}{\tikz \draw[-triangle 90] (0,0) -- +(.1,0);}
 \usepackage{latexsym}
 \usepackage{graphics}
 \usepackage{color}
\usepackage{lastpage}
\usepackage{fancyhdr}
\usepackage{multirow}
\allowdisplaybreaks
\usepackage{graphicx}
\graphicspath{ {F:/IMAGES/} }

 \newcommand{\F}{\mathcal{F}}

 \newcommand{\Supp}{\operatorname{Supp}}

  \newcommand{\reg}{\operatorname{reg}}

\newcommand{\proset}{\,\mathrel{\lower 4pt\hbox{$\scriptscriptstyle/$}
\mkern -14mu\subseteq }\,} 

 \newtheorem{theorem}{Theorem}[section]
 \newtheorem{corollary}[theorem]{Corollary}
 \newtheorem{lemma}[theorem]{Lemma}

 \newtheorem{proposition}[theorem]{Proposition}

 \newtheorem{question}{Question}
\usepackage{amsmath}

 \theoremstyle{definition}
 
 \newtheorem{remark}[theorem]{Remark}
 \newtheorem{definition}[theorem]{Definition}

 \newtheorem{example}[theorem]{Example}
\title[ Regularity  in weighted oriented graphs] { Regularity  in weighted oriented graphs}
\author[ M. Mandal and D.K. Pradhan ]{Mousumi Mandal$^*$ and Dipak Kumar Pradhan}
\address{Department of Mathematics, Indian Institute of Technology Kharagpur, 721302, India} \email{mousumi@maths.iitkgp.ac.in}
\address{Department of Mathematics, Indian Institute of Technology Kharagpur, 721302, India}\email{dipakkumar@iitkgp.ac.in}

\thanks{$^*$ Supported by SERB(DST) grant No.: $\mbox{EMR}/2016/006997$, India}

\begin{document}
\maketitle
 \begin{abstract}
 	Let $D$ be a weighted oriented graph with the underlying graph $G$ and $I(D), I(G) $ be the edge ideals corresponding to $D$ and $G$ respectively. We show that the regularity of edge ideal of a certain class of  weighted oriented graph remains same even after adding certain kind of new edges to it. We also establish the relationship between the regularity of edge ideal of weighted oriented path and cycle 
 	with the regularity of edge ideal of their  underlying graph 
 	when vertices of $V^+$ are sinks.
 \vspace*{0.2cm}\vspace*{0.2cm}\\  
 \noindent	AMS Classification 2010: 13D02, 13F20, 13C10, 05C22, 05E40, 05C20.

\end{abstract}

 \section{Introuduction} 
A  weighted oriented graph is a triplet $D = (V (D), E(D), w),$ where $V (D)$
is the vertex set, $E(D)$ is the edge set and $w$ is a weight function $ w : V (D) \longrightarrow \mathbb N^{+}$, where $\mathbb N^{+} = \{1, 2,\ldots \}$. Specifically $E(D)$ consists of ordered pairs of the form $(x_i,x_j)$ which represents a directed edge from the vertex $x_i$ to the vertex  $x_j$. 
The weight of a vertex $x_i \in V(D)$ is $w(x_i)$, denoted by $w_i$ or $w_{x_i}$. We set $V^+(D):= \{x\in V(D)~ |~ w(x) \geq 2  \} $ and it is denoted by $V^+$. The underlying graph of $D$ is the simple graph $G$ whose
vertex set is same as the vertex set of $D$ and whose edge set is $\{\{x, y\}|(x, y) \in E(D)\}$. If $V(D)=\{x_1,\ldots ,x_n\}$ we can regard each vertex $x_i$ as a variable and consider the polynomial ring $R=k[x_1,\ldots ,x_n]$ over a field $k$. The edge ideal of $D$ is defined as 
$$I(D)=(x_ix_j^{w_j}|(x_i,x_j)\in E(D)).$$
If a vertex $x_i$ of $D$ is a source, we shall always assume $w_i = 1$ because in  this case the definition of $I(D)$ does not depend on the weight of
$x_i.$ If $w(x)=1$ for all $x\in V$, then $I(D)$ recovers the usual edge ideal of the underlying graph $G$, which has been extensively studied in the literature in \cite{ali-2015,ali-2018,beyarslan1,bouchat,jacques,kiani}. 
The interest in edge ideals of weighted digraphs comes from coding theory, in the study of Reed-Muller types codes. The edge ideal of  weighted digraph appears as initial ideals of vanishing ideals of projective spaces over finite fields \cite{bernal}.
$\vspace*{0.2cm}$\\
  Algebraic invariants like Cohen-Macaulayness and unmixedness of edge ideals of weighted oriented graphs have been studied in \cite{gimenez,htha-2018,pitones}.
In \cite{pitones}, Pitones et al. have characterised the minimal strong property of $D$ when vertices of $V^+$ are sinks. Recently, the invariants like  Castelnuvo-Mumford regularity and projective dimension of weighted oriented graphs have drawn the attention of many researchers.  In \cite{zhu-2}, Zhu et al. have expressed the projective dimension and regularity of edge ideals of some class of weighted oriented forests or  cycles and in \cite{beyarslan}, Beyarslan et al. gave the formula for projective dimension and regularity of edge ideals of weighted oriented graphs  having the property $P$ defined as follows:\\ 
A weighted oriented graph $D$ is said to have property $P$ if there is at most one edge oriented into each vertex and suppose that for all non-leaf, non-source  vertices, $  x_j , $ either $ w_j \geq 2 $ or the unique edge $ (x_i, x_j ) $ into the vertex $ x_j $ has the property that $ x_i $ is a leaf.
$\vspace*{0.2cm}$\\
In general it is a difficult problem to give a general formula for the regularity of  edge ideal of an arbitrary weighted oriented graph even if the regularity of edge ideal of its  underlying graph is known as the edge ideal changes according to the  orientation of its edges and its weight function. In this paper we study the regularity of weighted oriented graphs arising by adding new edges to the weighted oriented  graphs having property $P$.
 By studying the regularity of edge ideal of  weighted oriented graphs  we partially answer  one question asked by H.T. H{\`a} in \cite{htha}.
 Also we  establish the relation between the regularity of edge ideal of weighted oriented graph $D$ and its underlying graph $G$ when $D$ is a   weighted oriented  path or cycle with vertices of $V^+$ are sinks.
$\vspace*{0.2cm}$\\
This paper is structured as follows. In 
section 2, we recall all the definitions and results that will be required for the rest of the paper. In section 3, we prove that the regularity of  edge   ideal   of one or more weighted oriented graphs with property $P$ remains unchanged even after adding new edges among the connected and disconnected components(Theorem \ref{exteded-s}). As some applications of Theorem \ref{exteded-s},  we compute the regularity of edge ideals of some weighted  oriented graphs whose underlying graphs are dumbbell graph, complete graph, join of two cycles and complete $m-$partite graph.  
In Theorem \ref{exteded-k}, we prove that the regularity of edge ideal of a weighted oriented graph with property $P$ remains same even after adding certain type of   oriented edges from new vertices  towards a single vertex of it. By using Proposition \ref{exteded-s.},  we able to give the combinatorial conditions    for  one question asked by H.T. H{\`a}  in \cite{htha}. 
In section 4, we compute the regularity of edge ideal of a weighted oriented path or cycle  in terms of regularity of edge ideal of their underlying graph  when vertices of $V^+$ are sinks.  
$\vspace*{0.2cm}$
\section{Preliminaries}
In this section we present some of the definitions and results that will be needed  throughout the paper.
 Let $D=(V(D), E(D), w)$ be a weighted oriented graph with underlying  graph $ G = (V(G), E(G)) $. 
For a vertex $ u $ in a graph $ G , $ let $ N_G(u) = \{v \in V(G) ~|~ \{u,v\} \in E(G)\} $ be the set of
neighbours of $ u   $ and set $N_G[u] := N_G(u)\cup \{u\}. $  For a subset $W \subseteq V(G) $ of the vertices in $ G, $ define $ G \setminus  W $ to
be the subgraph of $ G $ with the vertices in $ W $ (and their incident edges) deleted.
 Let $ x $ be a vertex of the weighted oriented graph $ D $, then the sets $ N_D^+ (x) =	\{y : (x, y) \in E(D)\}  $ and  $ N_D^- (x) = \{y : (y, x) \in E(D)\}  $ are called the out-neighbourhood and the in-neighbourhood of $ x $ respectively. Further, $ N_D(x) = N_D^+ (x)\cup N_D^- (x) $   is the set of neighbourhoods of $ x $ and set $N_D[u] := N_D(u)\cup \{u\}.$ 
For $T \subset V ,$ we
define the \textit{ induced  subgraph} $\mathcal{D} = (V( \mathcal{D}), E(\mathcal{D}), w)$ of $D$ on $T$ to be the
weighted oriented graph such that $V (\mathcal{D}) = T$ and for any
$ u, v \in V (\mathcal{D}),  $  $ (u,v) \in E(\mathcal{\mathcal{D}})$  if and only if
$(u,v) \in E(D)$.
Here $ \mathcal{D} = (V (\mathcal{D}), E(\mathcal{D}), w) $
is a weighted oriented graph with the same orientation  as in $D$ and for any $ u  \in V (\mathcal{D}), $
if $ u $ is not a source in $ \mathcal{D}, $ then its weight equals to the weight of $ u $ in $D,$ otherwise, its
weight in $ \mathcal{D} $ is $ 1. $ For a subset $W \subseteq V(D) $ of the vertices in $ D, $ define $ D \setminus  W $ to
	be the  induced subgraph of $ D $ with the vertices in $ W $ (and  their    incident edges) deleted.
For $ Y \subset E(D), $
we define $ D \setminus Y $ to be a   subgraph of $ D $ with all edges in $ Y $ deleted (but its vertices
remained). If $ Y = \{ e \} $ for some $ e \in E(D) ,$  we write $ D \setminus e $ in place of $ D \setminus \{e\}. $   Define $\deg_D(x) = |N_D(x)|$ for  $ x \in V(D) $. A vertex $ x \in V(D) $ is called a leaf vertex if $\deg_D(x)=1$. A vertex $ x \in V(D) $ is called a source vertex if $N_D (x)= N_D^+ (x) .$ A vertex $ x \in V(D) $ is called a sink vertex if $N_D (x)= N_D^- (x) .$ 
$\vspace*{0.2cm}$\\  
Now we give some algebraic definitions and results. Let $k$ be a field and $R=k[x_1,\ldots ,x_n]$ be the polynomial ring in $n$ variables over $k$. Suppose that $M$ is a non zero graded $R$-module with finite graded minimal free resolution $$0 \longrightarrow  \cdots   \longrightarrow \underset{j} \bigoplus  R(-j)^{\beta_{1,j}(M)} \longrightarrow \underset{j} \bigoplus  R(-j)^{\beta_{0,j}(M)}  \longrightarrow    M    \longrightarrow    0$$
 where  $ {\beta_{i,j}}(M) $ denote  the ($i,j$)-th graded Betti number of $ M ,$ is an invariant of $ M$ that equals the number of minimal generators of degree $ j $ in the $ i -$th syzygy module of $ M .$ The invariant which measures the complexity of the module is  Castelnuvo-Mumford regularity denoted by $\reg(M)$ and  defined as
 \begin{equation*}
 \reg(M) := \max \{ j -  i \mid  {\beta_{i,j}}(I)  \neq  0 \} .
 \end{equation*}
 Let $I \subset R$ be a monomial ideal. Then $\mathcal{G}(I)$ denotes the set of  minimal monomial generators of $ I $.
 In general, it is difficult to find the regularity even for monomial ideals. With the help of Betti splitting we can compute this type of invariant for certain class of ideals. The Betti splitting is defined as follows:
 \begin{definition}\label{2.1}  
 	Let $I$ be a monomial ideal and suppose that there exist monomial ideals $J$ and $K$ such that $\mathcal{G}(I)$ is the disjoint union of $\mathcal{G}(J)$   and $\mathcal{G}(K)$. Then
 	$I = J + K$ is a Betti splitting if
 	\begin{equation*}
 	{\beta_{i,j}}(I) =  {\beta_{i,j}}(J) +
 	{\beta_{i,j}}(K) +  {\beta_{i-1,j}}(J\cap K)
 	\end{equation*}
 	for all $i, j  \geq  0,$ 	where $ {\beta_{i-1,j}}(J\cap K) = 0 $ if $ i = 0.$
 \end{definition}
 This formula was first obtained for the total Betti numbers by Eliahou and Kervaire \cite{elia} and extended to the graded case by Fatabbi \cite{fatabbi}. In \cite{francisco}, the authors describe the following sufficient conditions for an ideal $I$ to have a Betti splitting.
 \begin{theorem}\cite[Corollary 2.7]{francisco} \label{betti.1}	Suppose that $I = J + K$ where $\mathcal{G}(J)$ contains
 	all the generators of $I$ divisible by some variable $x_i$ and $\mathcal{G}(K)$ is a nonempty set
 	containing the remaining generators of $I$. If $J$ has a linear resolution, then $I = J+K$
 	is a Betti splitting.
 \end{theorem}
 When $I$ is having a Betti splitting, Definition \ref{2.1} implies the following result:
 \begin{corollary}\label{betti.2}
 	If $I = J + K$ is a Betti splitting, then
 	$$ \reg (I)  = \max \{ \reg (J),~ \reg (K),~ \reg (J  \cap  K)  -  1 \} .$$
  \end{corollary}

 Let $u \in R$ be a monomial, we set $\Supp(u) = \{x_i : x_i  ~|~ u \}.$ Let $I$ be a monomial ideal,
 $\mathcal{G}(I) = \{u_1,\ldots, u_m\}$ denote the unique minimal set of monomial generators of $I$ and we set $\Supp(I) :=  \displaystyle{\bigcup_{i=1}^m}
 \Supp(u_i).$ The following  lemmas are well known. 
 
  \begin{lemma}\cite[Lemma 3.4]{zhu-1}  \label{reg.2}
 	Let $R_1 = k[x_1,\ldots , x_m]$ and $R_2 = k[x_{m+1}, \ldots, x_n]$ be
 	two polynomial rings,  $ I \subset $ $R_1$ and  $ J \subset $ $R_2$ be two nonzero homogeneous ideals. Then
 	\begin{enumerate}
 		\item[] $ \reg (I + J) = \reg (I) + \reg (J) - 1.$
 	\end{enumerate}
 \end{lemma}
 
 \begin{lemma}\cite[Lemma 2.3]{nv-838}\label{reg.3}
 	Let $I, J $ be two monomial ideals such that $\Supp(I) \cap  \Supp(J) = \phi .$ Then $ \reg (IJ) = \reg (I) + \reg (J).$

 \end{lemma}
 
 \begin{lemma}\cite[ Lemma 1.2]{nv-838}  \label{reg.4}  Let  $0  \rightarrow A \rightarrow B \rightarrow C \rightarrow 0$  be 	short exact sequence of finitely generated graded $R$-modules. Then
 		\item $\reg(B)$ $\leq \max \{\reg (A),~ \reg (C)\}$ and the equality holds if $\reg(A)$ $- 1  \neq$ $\reg(C).$ 
 	
 \end{lemma}
\begin{lemma}  \cite[Lemma 3.1]{htha}\label{H<G}
	Let $G = (V, E)$ be a simple graph. If $ G^\prime $ is an induced subgraph of $ G $, then $ \reg(I(G^\prime)) \leq  \reg(I(G)). $	
\end{lemma}
The following two corollaries are based on the regularity of edge ideal in path and cycle.   
\begin{corollary}\cite[Theorem 4.7]{beyarslan1}\label{path}
	Let $ G $ be a path of length $ n $ denoted as $ P_n .$ Then
	\item [(a)] $\displaystyle {\reg(I(P_n))= \floor*{\frac{n+2}{3}} + 1 } ,$  
	\item [(b)]  $\reg(I(P_n))=  \reg(I(P_{n-3})) + 1$ for $ n\geq 4. $  	
\end{corollary}  

\begin{corollary}\cite[Theorem 4.7, Theorem 5.2]{beyarslan1}\label{cycle}
	Let $ G $ be a cycle of length $ n $ denoted as $ C_n .$ Then
   \item [(a)] if $ n\equiv 0,1~(\mbox{mod}~3),$    then $\reg(I(G))=\reg(I(G  \setminus  \{x\}))= \reg(I(G \setminus N[x])) + 1 $ except $ n = 3,4$ and  $\reg(I(G))=\reg(I(G  \setminus  \{x\}))=2$ for $n=3,4$. 
	\item [(b)] if $ n\equiv 2~(\mbox{mod}~3),$   then $\reg(I(G))= \reg(I(G  \setminus \{x\})) + 1= \reg(I(G \setminus N[x])) + 1.$
	 
\end{corollary}
In order to deal with non square-free monomial ideals, polarization is proved to be a powerful
process to obtain a square-free monomial ideal from a given monomial ideal.
\begin{definition}
	Suppose that  $u={x_1}^{a_1}$ $\cdots$ $ {x_n}^{a_n} $ is a monomial in $R$. Then we define
	the polarization of $u$ to be the square-free monomial
	\begin{equation*}
	\mathcal{P} (u)  =  x_{11}x_{12} \cdots   x_{1a_1}x_{21}x_{22} \cdots x_{2a_2} \cdots  x_{n1}x_{n2} \cdots x_{na_n}
	\end{equation*}
	in the polynomial ring $R^\mathcal{P}=k[x_{ij}  \mid  1  \leq  i  \leq  n, 1  \leq  j  \leq  a_i ]$ . If  $ I \subset  R $ is a monomial
	ideal with $\mathcal{G}(I)$ = $ \{ u_1,\ldots, u_m   \}$, the polarization of $ I $, denoted by $I^\mathcal{P}$ is defined as:
	\begin{equation*}
	I^\mathcal{P} = (\mathcal{P} (u_1), \ldots ,\mathcal{P}(u_m))\end{equation*}
	which is a square-free monomial ideal in the polynomial ring $R^\mathcal{P}$ .
\end{definition}
The following lemma shows that the  regularity is preserved under polarization.
\begin{lemma}\cite[Corollary 1.6.3]{herzog}\label{pol} Let $I \subset R $ be a monomial ideal and $ I^\mathcal{P}  \subset R^\mathcal{P}$  its polarization. Then 
	
	\item[(a)]   $\beta_{ij} (I) = \beta_{ij} (I^\mathcal{P})$ for all $i$ and $j,$
	\item[(b)]  $\reg (I) = \reg ( I^\mathcal{P} ).$

\end{lemma}
Next we see  the connection of square-free monomial ideals with  hypergraphs and labeled hypergraphs.   
\subsection{Hypergraph}
 A hypergraph $ \mathcal{H} $ over $ X = \{x_1, \ldots , x_n \} $
 is a pair $ \mathcal{H} = (X, \mathscr{E}) $ where $ X $ is the set of elements called  vertices and $ \mathscr{E} $ is a set of non-empty subsets of $ X $ called hyperedges or edges. A hypergraph $\mathcal{H} $ is simple if there is no nontrivial containment between any
pair of its edges.\\
 The following construction gives a one-to-one correspondence between
square-free monomial ideals in $ R = k[x_1, \ldots , x_n]  $  and simple hypergraphs over $X.$ 
\begin{definition}
 Let $\mathcal{H} $ be a simple hypergraph on $ X. $ For a subset $ E \subset X ,  $ let $ x^E $ denotes
 the monomial $\displaystyle{\prod_{x_i \in E} x_i}$. Then the edge ideal of $\mathcal{H} $ is defined as
$$I(\mathcal{H}) = (x^E ~|~ E  \subseteq X ~\mbox{is an edge in}~\mathcal{H}
) \subset R. $$                                                                     
\end{definition}

\subsection{Labeled Hypergraph}
The labeled hypergraph associated to a given square-free
monomial ideal $ I $  introduced in \cite{McCullough}. In the definition of labeled hypergraph, generators of
the ideal correspond to vertices of the hypergraph and the edges of the hypergraph correspond to variables
which are obtained by the divisibility relations between the minimal generators of the ideal.
\begin{definition}\cite{McCullough}   Let $ I \subset R = k[x_1, \ldots , x_n] $ be a square-free monomial ideal with minimal monomial
generating set $ \{f_1, \ldots , f_{\mu}  \}. $ The labeled hypergraph of $ I $ is the tuple $ H(I) = (V ,  X,E,\mathcal{E} ). $   The
set $ V  = [{\mu}] $ is called the vertex set of $ H. $ The set $ \mathcal{E} $ is called the edge set of $ H(I) $ and is the
image of the function $ E : \{ x_1,\ldots , x_n \} \longrightarrow  \mathscr{P}(V)$ defined by $  E(x_i)   = \{j : x_i ~\mbox{divides}~ f_j \}  $ where $ \mathscr{P}(V) $
represents the
power set of $ V . $ Here the set $ X =  \{ x_i
: E(x_i) \neq   \emptyset   \} .$
\end{definition}  

The label of an edge $  F \in \mathcal{E}   $ is defined as the collection of variables $ x_i \in \{x_1, \ldots , x_n \}  $ such that $ E(x_i) = F. $
The number $ |X| $ counts the number of labels appearing in $ H(I) $ while $ |\mathcal{E}| $ counts the number of distinct edges. A vertex $ v \in V $ is
closed if $ \{v\} \in \mathcal{E} ,$ otherwise,  $  v    $  is open. An edge $ F \in \mathcal{E}  $ of $ H(I) $ is called simple if $ |F| \geq   2 $ and $ F $ has no proper subedges other than $\emptyset .$  If every open vertex is contained in exactly one simple edge, then we
say that $ H(I) $ has isolated simple edges.

\begin{example}
	Let $ I = (x_1x_3x_5, x_1x_2x_3, x_3x_4x_5,x_4x_5x_6) \subset k[x_1, \ldots , x_6]. $ Let $ f_1 = x_1x_3x_5, f_2 = x_1x_2x_3 ,f_3 = x_3x_4x_5 $   $ \mbox{and}$  $ f_4 = x_4x_5x_6.$ Then $ V  = \{ 1, 2, 3, 4   \},$ $ X =
	\{x_1, x_2, x_3, x_4,\\ x_5 ,x_6\} $  and    $ \mathcal{E} = \{ \{ 1,2\},\{2\} ,\{1,2,3\}, \{  3,4 \},   \{1,3,4\} , \{4\} \} .$ See Figure   1.    	
	\begin{figure}[!ht]
		\begin{tikzpicture}[scale=0.8]
	\definecolor{pastelgreen}{rgb}{0.47, 0.87, 0.47}
		\begin{scope}[ thick, every node/.style={sloped,allow upside down}]
		\draw[-,thick] (0,0) --(3,0);
		\draw[-,thick] (0,3) --(3,3);
			\shade[left color=gray!2!white,right color=gray]  (3,0) -- (0,3) -- (3,3)  -- cycle;  
			\shade[left color=blue!20!white,right color=blue]  (3,0) -- (0,3) -- (0,0)  -- cycle;	 
			\draw [fill=black] (0,0) circle [radius=0.15];
	\draw [fill=white] (3,0) circle [radius=0.15];
	\draw [fill=white] (0,3) circle [radius=0.15];
	\draw [fill=black] (3,3) circle [radius=0.15];
			\node at (-0.4,3.4) {$1$}; 
			\node at (3.4,3.4) {$2$};
			\node at (3.4,-0.4) {$3$};
			\node at (-0.4,-0.4) {$4$};
			\node at (1.5,3.4) {$x_1$}; 
			\node at (2,2) {$x_3$};
		    \node at (3.5,2.9) {$x_2$};
		    \node at (1.5,-0.4) {$x_4$};
		    \node at (1,1) {$x_5$};
		     \node at (-0.5,0.05) {$x_6$};  
		
		\end{scope}
		\end{tikzpicture}
		\caption{The labeled hypergraph of $ I = (x_1x_3x_5, x_1x_2x_3, x_3x_4x_5,x_4x_5x_6).$}  
	\end{figure}
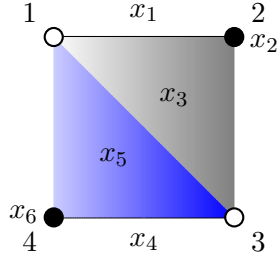
\end{example}   
 
\section{Some results of regularity in weighted Oriented graphs}    

In this section, we compute the
 regularity of $ R/I(D) $ for  certain class of  weighted oriented graph   $ D $ by connecting their polarized edge ideal with the labeled hypergraph and using the technique of Betti splitting. In this section we have considered a particular type of weighted oriented graph having  property $P$ as defined in the introuducion.

%

The regularity of edge ideal of weighted oriented graph having property $P$ was first studied by Beyarslan et al. in the following result.
\begin{proposition}\cite[Corollary 3.1]{beyarslan}\label{atmost-one}
Let $ D $ be a weighted oriented graph having
 property $P$ with weight function w on the   vertices $ {x_1,\ldots, x_n} $. 
Then $$\displaystyle{  \reg(R/I(D)) = \sum_{i=1}^{n} w_i -|E(D)|}  .$$

\end{proposition}
Beyarslan et al. have proved the above result using the concept of labeled hypergraph described in \cite{McCullough}. We noticed that the following result of Lin and  McCullough using the concept of  isolated simple edges of labeled hypergraphs will be useful for calculating the regularity of some new class of weighted oriented graphs.    
\begin{proposition}\cite[Theorem 4.12]{McCullough}\label{simple}
Let $  I \subset R  $ be a square-free monomial ideal and suppose that $ H(I) = (V, X, E, \mathcal{E}) $ has isolated simple edges. Then
 $$ \reg(R/I) = |X| - |V| +    \sum_{\substack{F \in \mathcal{E}\\F ~ \mbox{simple}}} (|F|-1) .$$      	
\end{proposition}
The following theorem shows that the regularity of  edge ideal of one or more weighted oriented graphs with property $P$ remains unchanged even after adding new edges among the connected and disconnected components.      
\begin{theorem}\label{exteded-s}
	Let $ D_1 , D_2 , \ldots, D_s $ for $s \geq 1$ are the  weighted   oriented graphs having
	 property $P$ with weight function $ w $  
	on vertex sets $ \{x_{1_1},\ldots,x_{{n_1}_1}\} ,$$ \{x_{1_2},\ldots,x_{{n_2}_2}\} , \ldots , \{x_{1_s},\\\ldots,x_{{n_s}_s}\} $ respectively.
	   Let  $ D   $ be a weighted oriented graph obtained by adding $ k $ new oriented edges among $ D_1 , D_2 , \ldots, D_s $   where every edge is of the form $ (x_{a_i}
	, x_{b_j} ) $     for some $ x_{a_i} \in  V(D_i),$  $x_{b_j}  \in V(D_j) $ $(i$ may equal with $j)$   with  $ w_{a_i} , w_{b_j} \geq 2 $ and  no vertex of  $ N^-_{D_j}(x_{b_j} ) $ is a leaf vertex in   $D_j.$            Then
	$$ \reg(R/I(D)) =\reg(R/I(D_1)) +\cdots+ \reg(R/I(D_s)).$$            
\end{theorem}

\begin{proof}
   	Here $ V(D)= V(D_1) \cup\cdots \cup V(D_s) =   \{ x_{1_1},\ldots,x_{{n_1}_1} ,  \ldots , x_{1_s},\ldots,x_{{n_s}_s} \} .$ 
Let $ |E(D_1)|=e_1 ,\ldots, |E(D_s)|=e_s, $  then   $ |E(D)|= e_1 + \cdots + e_s   + k.   $  Let $ I(D_1),\ldots,I(D_s),I(D) $ be the edge ideals of the weighted oriented graphs $ D_1,\ldots,D_s,  D$ respectively.
	Let $ m_1, \ldots , m_{e_1+\cdots + e_s+k} $ be the minimal
	generators of the polarized ideal $ I(D)^{\mathcal{P}} .$
	Suppose $ k_1,\ldots,k_s $ number of new  edges are oriented towards $ D_1,\ldots,D_s $ where $k_1+ \cdots +k_s = k. $       Let the $ k_1 $  new  edges are oriented towards   $ r_1 $   vertices of $ D_1 $ where for any  vertex $ x_{j_1} $ among  those $ r_1 $ verices  $ w_{j_1} \geq 2 $ and no vertex of  $ N^-_{D_1}(x_{j_1} ) $ is a leaf vertex in   $D_1.$    
	Now we consider the labeled hypergraph of $ I(D)^\mathcal{P} ,$ i.e., $ H(I(D)^\mathcal{P}) = (V ,  X,E,\mathcal{E})$ where $ V = [e_1 +\cdots +  e_s + k]. $ Without loss of generality let   $ x_{q_1} $ is one of those $ r_1 $ vertices and $ l_{1_1} $ number of new edges are oriented towards $ x_{q_1} .$ Since $D_1$ has property $P$, $ |N^-_{D_1}(x_{q_1})|=1 $ and so   $|N^-_{D}(x_{q_1})|= l_{1_1} + 1 .$  
	Let the generators corresponding to  those $ l_{1_1}  + 1 $  edges numbered as $ d_0,d_1,d_2,\ldots,d_{l_{1_1}}$  where each $ d_i \in [e_1 +\cdots +  e_s + k] $.
	Here $ E({x_{q_1i}}) = \{ d_0,d_1,d_2,\ldots,d_{l_{1_1}}   \}  $ for $ 2 \leq i \leq w_{q_1}. $ Let $ F_{1_1} =   E_{x_{{q_1}2}}.$  Then $ F_{1_1} \in  \mathcal{E} $ with label $ \{x_{q_1i}|2 \leq i \leq w_{q_1} \} $. 
	Since  no vertex of  $ N^-_{D_1}(x_{q_1}) $ is a leaf vertex in   $D_1,$ then no vertex of   $ N^-_{D}(x_{q_1}) $ is a leaf vertex  in $D.$ Thus there does not exist any element of $X$ which lies in the generators of $ I(D)^\mathcal{P} $ corresponding to some proper subset of $ F_{1_1} $ 
	 which implies $ F_{1_1} $ is a simple edge. 
	Let us  assume $ l_{2_1},\ldots,l_{{r_1}_1} $ number of new edges are oriented towards remaining $ r_1-1 $  vertices of $ D_1 $, then similarly  we get $ F_{2_1},\ldots,F_{{r_1}_1} $ are the simple edges with cardinality $ l_{2_1} + 1,\ldots,l_{{r_1}_1} + 1 $ respectively. Thus $ |F_{j_1}|=  l_{j_1} + 1 $ for $ 1 \leq j \leq r_1 $ where $ l_{1_1}+\cdots + l_{{r_1}_1} = k_1. $ 
	Let   $  k_i   $  new    edges are oriented towards   $ r_i $   vertices of $ D_i $ for $2 \leq i \leq s$ by the definition of new edges. If we assume $ l_{1_i},\ldots,l_{{r_i}_i} $ number of new edges are oriented towards  $ r_i $  vertices of $ D_i $, then similarly  we get $ F_{1_i},\ldots,F_{{r_i}_i} $ are the simple edges with   cardinality $ l_{1_i} + 1,\ldots,l_{{r_i}_i} + 1 $ respectively for $2 \leq i \leq s,$ i.e., $ |F_{j_i}|=  l_{j_i} + 1 $ for $ 1 \leq j \leq r_i $, $2 \leq i \leq s$   where $ l_{1_i}+\cdots + l_{{r_i}_i} = k_i $ for each $i.$    
	Let $ F = F_{1_1} \cup \cdots \cup F_{{r_1}_1}\cup \cdots \cup  F_{1_s} \cup \cdots \cup F_{{r_s}_s}  $ and $ C = V \setminus F. $
	Let $V_i \subset V$ be the set of vertices corresponding to the minimal generators of $ I(D_i)^\mathcal{P} $ for $1 \leq i \leq s$ in $ H(I(D)^\mathcal{P}) $.
	$\vspace*{0.2cm}$\\
	 For $ c \in C \cap V_1 , $ let $\displaystyle{ m_c= x_{{i_1}1}\prod_{t=1}^{w_{j_1}}x_{{j_1}t}}, $ i.e., a minimal generator of $ I(D)^\mathcal{P} $ corresponding to some edge $(x_{i_1},x_{j_1})$ of $D_1$. If $ x_{j_1} $ is a leaf in both $ D_1 $ and $ D $, then $ m_c $
	is the only minimal generator of $ I(D)^{\mathcal{P}} $
	which is divisible by $ x_{{j_1}1} $  and therefore $ \{c\} \in \mathcal{E}  $ with label $ \{x_{{j_1}t}|1 \leq t \leq w_{j_1} \}. $
	In case of $ x_{j_1} $ is a leaf in $D_1$ but not in $ D $, at least one new edge  is oriented away from $ x_{j_1} ,$   then  by definition of new edges $w_{j_1} \geq 2$ and $ m_c $
	is the only minimal generator of $ I(D)^{\mathcal{P}} $
	which is divisible by $ x_{{j_1}2} .$  Therefore $ \{c\} \in \mathcal{E} $ with label $ \{x_{{j_1}t}| 2 \leq t \leq w_{j_1} \}. $
	If $ x_{j_1} $ is not a leaf in $ D_1 $, then by assumption since $ x_{j_1} $ is not a source, either $ w_{j_1} \geq   2 $ or $ x_{i_1} $
	is a
	leaf in $D_1$.  If $  x_{i_1} $
	is a leaf in $D_1,$ then $w_{i_1}=1.$ Thus none of the new edges are connected with $x_{i_1}$ and   $x_{i_1}$ is a leaf in $D,$ then  $ m_c $ is the only minimal generator of $ I(D)^{\mathcal{P}} $ 
	which is divisible by $ x_{{i_1}1} $ and  $ \{c\} \in \mathcal{E}  $ with label $ \{x_{{i_1}1}\}. $ If $ x_{i_1} $ is not a leaf in $D_1$ then $ w_{j_1} \geq 2 $ in $D_1$  and so is in $D.$ By the property $P$ of $D_1$,  at most one edge is oriented into the vertex $ x_{j_1} $ in $D_1$ and so is in $D$ because no new edge is oriented towards $x_{j_1}$.  Then  $ m_c $ is divisible by $ x_{{j_1}2} $ and none of  any other generator of $ I(D)^{\mathcal{P}} $ is  divisible by  $ x_{{j_1}2} .$  Thus $ \{c\} \in \mathcal{E}  $ with label $ \{x_{{j_1}t}| 2 \leq t \leq w_{j_1} \}. $ Therefore for every $ c \in C \cap V_1,$ $ \{c\} \in \mathcal{E} .$  By the similar arguement for every $ c \in C \cap V_i,$ $ \{c\} \in \mathcal{E} $ where $ 2 \leq i \leq s  $. So every $c \in C$ is closed.  
	Here   each of the remaining edges of $ \mathcal{E} $ is  some image   $  E({x_{p_i1}})   $ where $ x_{p_i} $ is one of the non-leaf vertex  of $D_i$ for some $i \in [s],$ $p \in [n_i]$  and it  contains either  one $ F_{{j_i}} $ for some $j \in [r_i]$  or at least one \{c\} for some $c \in  C \cap V_i  $ as a proper subset. Thus they are not simple. Therefore $F_{{j_i}} $ are the only  simple edges in the labeled hypergraph $ H(I(D)^{\mathcal{P}})  $ and  by the definition of $F_{{j_i}}$'s no two $F_{{j_i}}$'s have a common element  which implies every open vertex is contained in exactly one simple edge, i.e.,  $H(I(D)^{\mathcal{P}})$ has isolated simple edges. Hence by Lemma \ref{pol}, Proposition \ref{simple} and Proposition \ref{atmost-one}, we have
	\begin{align*}
	 \reg(R/I(D))&=\displaystyle{ |X| - |V| +    \sum_{i=1}^{r_1} (|F_{i_1}|-1) + \cdots +  \sum_{i=1}^{r_s} (|F_{i_s}|-1) }\\
	 &= \displaystyle{   \sum_{v \in V(D_1)}w(v) + \cdots   +   \sum_{v \in V(D_s)}w(v)         -(e_1 +\cdots+ e_s + k  )}\\
	 	  &\hspace*{0.35cm}+ (l_{1_1}+\cdots + l_{{r_1}_1})
+ \cdots +  (l_{1_s}+\cdots + l_{{r_s}_s})\\
   &= \displaystyle{   \sum_{v \in V(D_1)}w(v) + \cdots +     \sum_{v \in V(D_s)}w(v)       -(e_1 + \cdots +  e_s + k  ) + k_1 + \cdots + k_s}\\
   &= \displaystyle{  \sum_{v \in V(D_1)}w(v) - e_1 + \cdots +   \sum_{v \in V(D_s)}w(v) - e_s } \\
   &= \reg(R/I(D_1)) + \cdots + \reg(R/I(D_s)).
  \end{align*}

\end{proof}

\begin{corollary}\label{exteded-1}
	Let $ D $ be a weighted oriented graph having
	property $P$ with weight function w on the vertices $ {x_1, \ldots   , x_n} $. Let $ D^\prime $ be a weighted oriented graph obtained by adding $ k $ new oriented edges where each edge is of the form $ (x_i,x_j) $ for some $ x_i,x_j \in V(D) $ with $ w_i ,$  $ w_j \geq 2 $ and  no vertex of  $ N^-_{D}(x_j) $ is a leaf vertex in   $D.$   
	Then $$\displaystyle{  \reg(R/I(D^\prime))= \reg(R/I(D)) } .$$  
\end{corollary}
\begin{proof}
 The proof directly follows from Theorem \ref{exteded-s} for $s=1.$
\end{proof}

In the next two corollaries we give application of Corollary \ref{exteded-1} into some particular kind of weighted oriented graphs.\\
A graph $G$ is called a dumbbell graph if $G$ contains two cycles $C_n$ and $C_m$ of length $ n $ and $ m $ respectively joined by a path $P_r$ of length $r$ and we denote it by $C_n\cdot P_r \cdot C_m.$\\
A path or cycle is said to be naturally oriented if	
all of its edges oriented in same direction.   In a naturally oriented unicyclic graph, the cycle is naturally oriented and each edge of the tree connected with the cycle oriented away from the cycle.
A naturally oriented dumbbell graph is the union of two naturally oriented cycles and a naturally oriented path joining them.
\begin{corollary}\label{3.6}
	Let $D^{\prime} = (V(D^{\prime} ),E(D^{\prime} ),w )$ be a weighted naturally oriented dumbbell graph whose underlying graph is $G = C_n\cdot P_1\cdot C_m$ where $ C_n=x_1\ldots x_nx_1   ,P_1=x_1y_1$ and $C_m= y_1\ldots y_my_1$ with   $ w(x) \geq 2 $ for any vertex $ x .$  Then   $$ \reg(R/I(D^\prime))  \displaystyle{=\sum_{x \in V(D^\prime)}w(x) -  |E(D^{\prime} )| + 1   }.   $$
\end{corollary}

\begin{proof}

	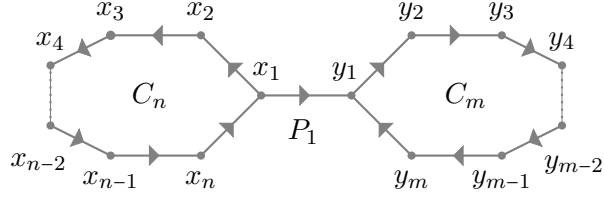
\begin{figure}[!ht]  
		\begin{tikzpicture}[scale=0.8]
		\begin{scope}[ thick, every node/.style={sloped,allow upside down}]
		\draw [fill, gray](1,1)-- node {\midarrow} (0,0.5);
		\draw [fill, gray](2.5,1)-- node {\midarrow} (1,1);
		\draw [fill, gray](3.5,0)-- node {\midarrow} (2.5,1);
		\draw [fill, gray][dotted,thick] (0.0,-0.5) --(0.0,0.5);
		\draw [fill, gray](0,-0.5) --node {\midarrow}(1,-1);
		\draw [fill, gray](1,-1) --node {\midarrow}(2.5,-1);
		\draw [fill, gray](2.5,-1) --node {\midarrow}(3.5,0);
		\draw [fill, gray][fill] (0,0.5) circle [radius=0.05];
		\draw [fill, gray][fill] (0,-0.5) circle [radius=0.05];
		\draw [fill, gray][fill] (1,1) circle [radius=0.06];
		\draw [fill, gray][fill] (2.5,1) circle [radius=0.05];
		\draw [fill, gray][fill] (3.5,0) circle [radius=0.05];
		\draw [fill, gray][fill] (1,-1) circle [radius=0.05];
		\draw [fill, gray][fill] (2.5,-1) circle [radius=0.05];
		\draw [fill, gray](3.5,0) --node {\midarrow}(5,0);
		\draw [fill, gray](5,0) --node {\midarrow}(6,1);
		\draw[thick,gray] (6,-1) --node {\midarrow}(5,0);  
		\draw [fill, gray](6,1) --node {\midarrow}(7.5,1);
		\draw [fill, gray] [dotted,thick] (8.5,-0.5) --(8.5,0.5);
		\draw [fill, gray](7.5,1) -- node {\midarrow}(8.5,0.5);
		\draw [fill, gray](8.5,-0.5) --node {\midarrow}(7.5,-1);
		\draw [fill, gray](7.5,-1) --node {\midarrow}(6,-1);
		\draw [fill, gray][fill] (6,1) circle [radius=0.05];
		\draw [fill, gray][fill] (7.5,1) circle [radius=0.05];
		\draw [fill, gray][fill] (8.5,0.5) circle [radius=0.05];
		\draw [fill, gray][fill] (8.5,-0.5) circle [radius=0.05];
		\draw [fill, gray][fill] (7.5,-1) circle [radius=0.05];
		\draw [fill, gray][fill] (6,-1) circle [radius=0.05];
		\draw [fill, gray][fill] (5,0) circle [radius=0.05];
		\node at (0,0.9) {$x_4$};
		\node at (3.6,0.4) {$x_1$};
		\node at (1,1.4) {$x_3$};
		\node at (2.5,1.4) {$x_2$};
		\node at (-0.2,-1.1) {$x_{n-2}$};
		\node at (1,-1.4) {$x_{n-1}$};
		\node at (2.5,-1.4) {$x_n$};
		\node at (8.5,0.9) {$y_4$};
		\node at (4.9,0.4) {$y_1$};
		\node at (6,1.4) {$y_2$};
		\node at (7.5,1.4) {$y_3$};
		\node at (8.7,-1.15) {$y_{m-2}$};
		\node at (6,-1.4) {$y_{m}$};
		\node at (7.5,-1.4) {$y_{m-1}$};
		\node at (1.65,0) {$C_n$};
		\node at (6.9,0) {$C_m$};
		\node at (4.2,-0.6) {$P_1$};
		\end{scope}
		\end{tikzpicture}
		\caption{Weighted naturally oriented Dumbbell graph($G= C_n \cdot P_1 \cdot C_m$).}
		\label{fig.2}
	\end{figure}
	Here $ V(D^\prime)=\{x_1,\ldots,x_n,y_1,\ldots,y_m\} .$	
Without loss of generality we give orientation to $ D^\prime $ as shown in Figure \ref{fig.2}.
Let $ D = D^\prime\setminus e    $ where $ e=$ $(y_m,y_1). $ Since $ D $ is a weighted naturally oriented unicyclic graph, it has property $P.$ Thus by Proposition  \ref{atmost-one}, we have $ \reg(R/I(D))= \displaystyle{\sum_{x \in V(D)}w(x) -  |E(D )|} = \displaystyle{\sum_{x \in V(D^{\prime})}w(x)} - ( |E(D^{\prime} )| - 1 ).$  By adding  the oriented  edge $e$ to $ D $ we get $ D^\prime  .$ Hence 
by Corollary \ref{exteded-1}, $ \reg(R/I(D^\prime))= \reg(R/I(D))  \displaystyle{=\sum_{x \in V(D^\prime)}w(x) -  |E(D^{\prime} )| + 1   }. $  
 \end{proof}

\begin{remark}
	Similarly we can find the regularity of edge ideal of weighted naturally oriented dumbbell graph  when the two naturally oriented  cycles are joined by a naturally oriented path of length $ r $ for $ r \geq 2 . $
\end{remark}

\begin{corollary}\label{4.1}
	Let $ D $  be a weighted naturally oriented cycle  whose underlying graph is $C_n=x_1\ldots x_nx_1$  with $ w(x) \geq 2$ for any vertex $ x.$ Let $ D_k $ be a weighted oriented graph we get after addition of $k$ diagonals  in any direction  to $ D $ for $1 \leq k \leq   \binom{n}{2} - n  $ and here $ D_{\binom{n}{2} - n} $ is a weighted oriented complete graph. Then for each $k,$
	$$ \reg(R/I(D_k))= \reg(R/I(D))  = \displaystyle{\sum_{i=1}^n{w_i} - n  }. $$    
\end{corollary}
\begin{proof}
	Here $V(D_k)= V(D)=\{x_1,\ldots,x_n\} $ for each $k$.	Since $ D $ is a weighted naturally oriented cycle, 
	 it has property $P$. Thus by Proposition \ref{atmost-one}, $ \reg(R/I(D)) = \displaystyle{\sum_{i=1}^n{w_i} - n}  .$  Here $D_k$ is obtained by adding $k$ diagonals  with any direction  to $ D $ for $1 \leq k \leq   \binom{n}{2} - n.$ Hence by Corollary \ref{exteded-1}, we have $  \reg(R/I(D_k))= \reg(R/I(D))  = \displaystyle{\sum_{i=1}^n{w_i} - n  } $  for each $k.$

\end{proof}
As some application of Theorem \ref{exteded-s}, we derive the formulas for regularity of edge ideals of some weighted oriented graphs whose underlying graphs are the join of two cycles and complete $ m $-partite graph.
\vspace*{0.2cm}\\
The join of two simple graphs $ G_1 $ and $ G_2 $, denoted by  $ G_1 * G_2 $ is a graph on the vertex set $ V(G_1) \sqcup  V(G_2) $ and  edge set   $ E(G_1) \cup E(G_2) $ together with all the edges joining $V(G_1)$ and $V(G_2).$
\vspace*{0.2cm}\\  
A graph $ G $  is   $m-$partite graph if 
$V(G)$ can be partitioned into $m$ independent sets $ V_1 , \ldots,V_m  ,$ in such a way that any
edge of graph $ G $ connects vertices from different subsets.\\
A complete $m-$partite graph is a $m-$partite graph such that there is an edge between
every pair of vertices that do not belong to the same independent set.\\   
\begin{corollary}\label{4.4}
	Let $ D_1 $ and $ D_2 $ be two  weighted naturally oriented cycles whose   underlying graphs are  $C_n = x_1\ldots x_nx_1$  and  $C_m = y_1\ldots y_my_1$ respectively  with $ w(v) \geq 2 $ for any vertex $ v .$
	Let $ D^\prime_k $ be a weighted oriented graph we get after addition of $k$ oriented edges joining $V(G_1)$ and $V(G_2)$  in any direction between $ D_1 $ and $D_2$ for $1 \leq k \leq  mn  $ and here  $ D^\prime_{mn} $ is a weighted oriented graph whose   underlying graph  is $ C_n * C_m  .$  Then for each $k,$
$$ \reg(R/I(D^\prime_k)) = \reg(R/I(D_1)) + \reg(R/I(D_2)) = \displaystyle{ \sum^n_{i=1}w_{x_i} + \sum^m_{i=1}w_{y_i} -  (n+m) .  }   $$            
\end{corollary}

\begin{proof}
	Here $ V(D^\prime_k)=V(D_1)  \cup V(D_2) = \{x_1,\ldots,x_n,y_1,\ldots,y_m\} $  for $1 \leq k \leq mn$. Since $ D_1 $ and $D_2$ are  weighted naturally oriented cycle, they have property $P$. Thus
	by Proposition \ref{atmost-one}, $ \reg(R/I(D_1)) = \displaystyle{\sum^n_{i=1}w_{x_i} - n}  $ and  $ \reg(R/I(D_2)) = \displaystyle{\sum^m_{i=1}w_{y_i} - m}  $. Here  $ D^\prime_k $ is obtained by  adding $k$  new oriented    edges joining  $V(C_n)$ to $V(C_m)$ in any direction  between $ D_1 $ and $ D_2   $   for $1 \leq k \leq mn.$ Hence  by Theorem \ref{exteded-s} for $s=2$, we have  $ \reg(R/I(D^\prime_k)) =\reg(R/I(D_1)) + \reg(R/I(D_2)) = \displaystyle{ \sum^n_{i=1}w_{x_i} + \sum^m_{i=1}w_{y_i}     }    -  (n+m)  $  for each $k.$ 
\end{proof}
In the following corollary, we give a short proof of \cite[Theorem 5.1]{zhu-3} using Theorem \ref{exteded-s}.
\begin{corollary}
Let   $ D = (V(D), E(D), w) $ is a
weighted oriented complete $ m $-partite graph for   $ m \geq  3 $ with vertex set $ V(D) = \displaystyle{ \bigsqcup_{i=1}^m V_i  }  $ and edge set
$ E(D) = \displaystyle{ \bigsqcup_{i=1}^m E(D_i ) }    $   where $ D_i $
is a weighted oriented complete bipartite graph  on $ V_i \sqcup V_{i+1} $ and every edge of  $ E(D_i) $ is of the form $(u,v)$ with $u \in V_i$, $v \in V_{i+1}$  for
 $ 1 \leq  i \leq m $ by setting
$ V_{m+1} = V_1 $. 
If $ w(x) \geq  2 $ for all $ x \in V(D)  $, then
$$ \reg(R/I(D))= \displaystyle{\sum_{x \in V(D)}w(x) -  |V(D )|} .$$

\end{corollary}
\begin{proof}
Let $V_i =  \{ {x_1}_i,{x_2}_i,\ldots, {x_{n_i}}_i \}    $ for $ 1 \leq i \leq m$.
  For $ 1 \leq i \leq m$, let $D^\prime_i$ be the  oriented graph over  vertex set $V(D^\prime_i) = V_i \sqcup V_{i+1} $ and the edge set $E(D^\prime_i) = \{ ({x_1}_i,{x_1}_{i+1}), ({x_2}_i,{x_2}_{i+1}),\\\ldots,({x_{n_i}}_i,{x_{n_i}}_{i+1})     \} \cup  \{({x_1}_i,{x_{n_i + 1}}_{i+1}), ({x_1}_i,{x_{n_i + 2}}_{i+1}) ,\ldots,({x_1}_i,{x_{n_{i+1}}}_{i+1})  \}    $ if $n_i < n_{i+1} $
or the edge set  $E(D^\prime_i) = \{ ({x_1}_i,{x_1}_{i+1}) ,({x_2}_i,{x_2}_{i+1}),\ldots,({x_{n_{i+1}}}_i,{x_{n_{i+1}}}_{i+1})     \}          $ if $n_i \geq  n_{i+1} $.\\
Let   $ D^\prime = (V(D^\prime), E(D^\prime), w) $ be the
weighted oriented $ m- $partite graph  over the vertex set $ V(D^\prime) = \displaystyle{ \bigsqcup_{i=1}^m V_i  }  $ and the edge set
$ E(D^\prime) = \displaystyle{ \bigsqcup_{i=1}^m E(D^\prime_i ) }    $ with the same weight function as in $D$. Observe that in each $D^\prime_i$, there is exactly one edge oriented into each vertex of $V_{i+1}$  which implies in $D^\prime,$ exactly one edge oriented into each vertex of $ V(D^\prime) $. Thus each component of $D^\prime$ is with property $P$. Hence by Proposition \ref{atmost-one}, $  \reg(R/I(D^\prime))= \displaystyle{\sum_{x \in V(D^\prime)}w(x) -  |E(D^\prime )|}  =  \displaystyle{\sum_{x \in V(D^\prime)}w(x) -  |V(D^\prime )|}  = \displaystyle{\sum_{x \in V(D)}w(x) -  |V(D )|} .$   Here $D$ is obtained by adding all the  edges of the set $E(D)\setminus E(D^\prime)$ to $D^\prime$.  If there is $s$ components  in $D^\prime$ for some $s \geq 1$, then by Theorem \ref{exteded-s}, we have $$ \reg(R/I(D))= \reg(R/I(D^\prime))= \displaystyle{\sum_{x \in V(D)}w(x) -  |V(D )|} .$$    

\end{proof}  

In the following theorem, we show that the regularity of edge ideal of a weighted oriented graph $D$ with property $P$ remains same even after adding certain type of   edges  from new vertices  oriented towards a single vertex of  it.    

\begin{theorem}\label{exteded-k}
	Let $ D $ be a weighted oriented graph having
	 property $P$ with weight function $ w $ on the vertices $ {x_1, \ldots   , x_n} $.
		Let $ D^\prime_k $ be a weighted oriented graph after adding $ k $ new oriented edges to $ D $ at $x_p$ with $ w_p \geq 2 $ for a fixed $ p \in [n] $    where each edge is of the form  $  (x_{n+i},x_p) $ for some $ i \in [k] $ and each $ x_{n+i} $ is a new vertex. Then 
	$$\displaystyle{  \reg(I(D^\prime_k))= \reg(I(D)) = \sum_{i=1}^{n} w_i -|E(D)|} + 1  .$$ 
\end{theorem}

\begin{proof} 
	
Here $ V(D) = \{ x_1,\ldots ,x_n\} $.   Without loss of generality let $x_p = x_n$.  
We prove this theorem by applying induction on the number of new oriented edges added to $D$ at $x_n$.\\ 
 Base case: If  $ k=0 $, then the proof follows trivially.\\ 
 For $k \geq 1$, let $ D^\prime_k $ be a weighted oriented graph after adding the $k$ new oriented edges $  (x_{n+1},x_n)  ,(x_{n+2},x_n),$\ldots,$  (x_{n+k},x_n)  $ from new vertices   to  $x_n$ in $D$ where $ w_n \geq 2 $.     
Here $ I(D^\prime_k) = I(D^\prime_{k-1}) + x_{n+k}x_n^{w_n}          $
where     $D^\prime_{k-1}=D^\prime_k\setminus\{x_{n+k}\} $. Then $ I(D^\prime_k)^{\mathcal{P}} = I(D^\prime_{k-1})^{\mathcal{P}} + \displaystyle{ x_{n+k,1}\prod_{j=1}^{w_n}x_{nj}.} $ Note that in $ D^\prime_{k-1} ,$   there are $ k-1 $ new oriented edges added to $ D $ at $x_n$.   
Let $\displaystyle{ J =     x_{n+k,1}\prod_{j=1}^{w_n}x_{nj}}$ and $ K = I(D^\prime_{k-1})^{\mathcal{P}}. $ Since $ J $ has linear resolution, $ I(D^\prime_k)^{\mathcal{P}} = J + K $ is a Betti splitting. Here $\reg(J)   = w_n + 1  .$ By Lemma \ref{pol}, Proposition \ref{atmost-one} and induction hypothesis, we have $ \reg(K) =\reg(I(D^\prime_{k-1})) =\reg(I(D)) = \displaystyle{\sum_{x \in V(D)}} w(x) $ $  -|E(D)| + 1.$ Now we want to compute $\reg(J \cap K)-1$.\\ 
\begin{figure}[!h]
	\begin{tikzpicture}[scale=0.9]
	\begin{scope}[ thick, every node/.style={sloped,allow upside down}]
	\definecolor{umber}{rgb}{0.39, 0.32, 0.28}
    \definecolor{taupe}{rgb}{0.28, 0.24, 0.2}
    \definecolor{usccardinal}{rgb}{0.6, 0.0, 0.0}
	\draw[dotted,thick] (3.6,0.5) --(3.9,0.9);
	\draw[dotted,thick] (3.5,-0.25) --(3.7,-0.7);
	\draw[dotted,thick] (6.3,-1.2) --(6.6,-1.9);
	\draw[dotted,thick] (6.6,-2.6) --(6.4,-3);
	\node at (2.9,-1.3) {$x_{n_s}$};
	\node at (2.6,0.7) {$x_{n_r}$};
	\node at (2.4,-0.3) {$x_{n_{r+1}}$};
	\node at (3.6,1.6) {$x_{n_2}$};
	\node at (4.6,2) {$x_{n_1}$};
	\draw (7.5,0.5) --node {\midarrow}(5,0);
	\draw (5,-2) --node {\midarrow}(5,0);  
	\draw (5,0) --node {\midarrow}(4.5,1.7);
	\draw (5,0) --node {\midarrow}(3.7,1.3);  
	\draw (5,0) --node {\midarrow}(3,-0.2);
	\draw (5,0) --node {\midarrow}(3.3,-1.2);
	\draw (5,0) --node {\midarrow}(3,0.5);
	\draw [fill] (7.5,0.5) circle [radius=0.05];
	\draw [fill] (5,0) circle [radius=0.05];
	\draw [fill] (5,-2) circle [radius=0.05];
	\draw [fill] (4.5,1.7) circle [radius=0.05];
	\draw [fill] (3.7,1.3) circle [radius=0.05];
	\draw [fill] (3,-0.2) circle [radius=0.05];
	\draw [fill] (3.3,-1.2) circle [radius=0.05];
	\draw [fill] (3,0.5) circle [radius=0.05];
	\node at (6.4,-0.4) {$x_{{n-1}_1}$};
	\node at (7.2,-0.9) {$x_{{n-1}_2}$};
	\node at (7.7,-1.9) {$x_{{n-1}_p}$};  
	\node at (7.9,-2.7) {$x_{{n-1}_{p+1}}$};
	\node at (7.7,-3.7) {$x_{{n-1}_t}$};
	\node at (4.4,-2) {$x_{{n-1}}$};
	\node at (3.9,-3.9) {$x_{{n-2}}$};
	\draw (5,-2) --node {\midarrow}(5.8,-0.5);
	\draw (5,-2) --node {\midarrow}(6.5,-0.9);  
	\draw (5,-2) --node {\midarrow}(7,-1.9);
	\draw (5,-2) --node {\midarrow}(7,-3.7);
	\draw (5,-2) --node {\midarrow}(7,-2.6);
	\draw[thick, usccardinal] (4.5,-3.8) --node {\midarrow}(5,-2);    
	\draw [fill] (5.8,-0.5) circle [radius=0.05];
	\draw [fill] (6.5,-0.9) circle [radius=0.05];
	\draw [fill] (7,-1.9) circle [radius=0.05];
	\draw [fill] (7,-3.7) circle [radius=0.05];
	\draw [fill] (7,-2.6) circle [radius=0.05];
	\draw [fill] (7.3,1.1) circle [radius=0.05];
	\draw [fill] (6.8,1.8) circle [radius=0.05];
	\draw [fill, usccardinal] (4.5,-3.8) circle [radius=0.05];    
	\node at (8.2,0.5) {$x_{n+1}$};
	\node at (8,1.2) {$x_{n+2}$};
    \node at (7.3,2.2) {$x_{n+k}$};  
	\node at (5.2,0.6) {$x_n$};  
	
	\draw (7.3,1.1) --node {\midarrow}(5,0); 
	\draw (6.8,1.8) --node {\midarrow}(5,0);
	\draw[dotted,thick] (6.7,0.9) --(6.4,1.3);       
	\end{scope}
	\end{tikzpicture}
	\caption{Neighbourhood of $x_{n-1}$ and $x_n$ in weighted oriented graph $D^\prime_k.$}  
\end{figure}
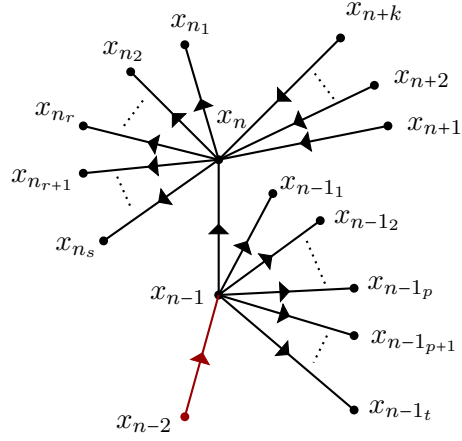
Let  $N_{D^\prime_k}^- (x_n) = \{ x_{{n-1}}, x_{{n+1}}, x_{{n+2}},\ldots, x_{{n+k}} \} $ where   $x_{n-1} \in V(D)$ and $ x_{{n+1}}, x_{{n+2}},\ldots, x_{{n+k}} $ are the new vertices in $D^\prime_{k}$.       
Let  $ N_{D}^+ (x_n)  = \{x_{{n_1}},x_{{n_2}},\ldots,x_{{n_r}},x_{n_{r+1}},\ldots,x_{n_s} \} $  among which $ w_{n_i} = 1$ for $ 1 \leq i \leq r $ and  $ w_{n_i} \geq 2 $ for $ r+1 \leq i \leq s $ in $D$. Let   $ N_{D}^+( x_{n-1} ) = \{x_n,x_{{n-1}_1},x_{{n-1}_2},\ldots,x_{{n-1}_p},x_{{n-1}_{p+1}},\ldots,x_{{n-1}_{t}} \}$ such that  $ x_{{n-1}_1},x_{{n-1}_2},\ldots,x_{{n-1}_p}$ are leaf vertices and $ x_{{n-1}_{p+1}},x_{{n-1}_{p+2}},\ldots,x_{{n-1}_{t}} $ are non-leaf vertices in $D$.  Here the $r$ vertices $x_{{n_1}},x_{{n_2}},\ldots,x_{{n_r}}$ are leaf vertices and the $t-p$ vertices $x_{{n-1}_{p+1}},x_{{n-1}_{p+2}},\ldots,x_{{n-1}_{t}}$ are of weight $ \geq 2$  in $D$ by the property $P.$\\
Let $\displaystyle{ J  \cap K = JL =    ( x_{n+k,1}\prod_{j=1}^{w_n}x_{nj}  )  ((x_{{n+1},1},x_{{n+2},1},\ldots,x_{{n+k-1},1},x_{{n-1},1},x_{{n_1},1},x_{{n_2},1},\ldots,x_{{n_r},1},} \\ \displaystyle{\prod_{j=1}^{w_{{n_{r+1}}}} x_{n_{r+1},j}   ,  \prod_{j=1}^{w_{{n_{r+2}}}} x_{n_{r+2},j} ,\ldots,\prod_{j=1}^{w_{{n_{s}}}} x_{n_{s},j}   ) + (I(D\setminus\{x_{n},x_{n-1}\})^{\mathcal{P}}).}$
Let $ L_1 =\displaystyle{ (\prod_{j=1}^{w_{{n_{r+1}}}} x_{n_{r+1},j} ,}\\\displaystyle{\prod_{j=1}^{w_{{n_{r+2}}}} x_{n_{r+2},j} ,\ldots,\prod_{j=1}^{w_{{n_{s}}}} x_{n_{s},j})  }$ and $ \displaystyle{L_2 =   (I(D\setminus\{x_{n},x_{n-1}\}))^{\mathcal{P}}}.$ Note that
    $ \reg(L) = \reg(L_1 + L_2 ) . $ By expressing $L_1$ as  $\displaystyle{ (x_{n_{r+1},1}\prod_{j=2}^{w_{{n_{r+1}}}} x_{n_{r+1},j} ,x_{n_{r+2},1}\prod_{j=2}^{w_{{n_{r+2}}}} x_{n_{r+2},j}\\ ,\ldots,x_{n_{s},1}\prod_{j=2}^{w_{{n_{s}}}} x_{n_{s},j})}$,   we can regard  $ L_1 + L_2  $ as the polarized    edge ideal  of  the  weighted oriented graph with $|E(D)| - (t+r+2)$ edges obtained from  $  D\setminus\{x_{n},x_{n-1}\} $
         by adding one leaf of weight $ w_{n_i} - 1 $  to each $x_{n_i}$   for $ i = r + 1,\ldots, s $. Observe that in this graph the $s-r$ vertices   $x_{n_{r+1}},\ldots,x_{n_s} ,$  the $t-p$ vertices $ x_{{n-1}_{p+1}}, \ldots,x_{{n-1}_{t}} $ become  source vertices and each of its component is with property $P.$ \\  
       So   we can apply Proposition \ref{atmost-one} to compute the  $\reg(L_1 + L_2  )  .  $\\ 
  Case-A: Assume $   N_D^- (x_{n-1}) \neq \emptyset $.\\ By the property $P$ of $D,$ $|N_{D}^- ( x_{n-1})| = 1  .$ Let $N_{D}^- ( x_{n-1}) =  \{x_{{n-2}}\} .$ 
  \\ 
  Case-I: Let  $ x_{n-2}\in  N_D^+ (x_n) $. Then by the property $P$,    $x_{n-2} \in  \{ x_{n_{r+1}},x_{n_{r+2}},\ldots,x_{n_s}      \} $ and $w_{n-1} \geq 2.$          
Thus by Lemma \ref{reg.3} and   Proposition \ref{atmost-one}, we have  
\begin{align*}
\reg(J\cap K)-1 &= \reg(J) + \reg(L) -1\\
\hspace*{2.43cm} &= \reg(J) +  \reg(L_1 + L_2  )-1 \\
\hspace*{2.43cm} &= (w_n + 1) + \displaystyle{\sum_{x \in V(D) \setminus V_1} }   w(x) + (1+w_{n_{r+1}}-1) + (1+w_{n_{r+2}}-1)\\
&\hspace*{0.4cm}+ \cdots + (1+w_{n_{s}}-1) +(t-p) - [|E(D)| - (t+r+2)] + 1 -1 \\
&= \displaystyle{ \sum_{x \in V(D) \setminus V_2}} w(x)   -|E(D)| + (t-p) + (t+r+2) + 1     
\end{align*}  
where $V_1 =  \{x_n,x_{n-1},x_{{n_1}},x_{{n_2}},\ldots,x_{{n_r}},x_{n_{r+1}},\ldots,x_{n_s},x_{{n-1}_1},\ldots,x_{{n-1}_p},x_{{n-1}_{p+1}}, \ldots,x_{{n-1}_{t}}\}$
\\ and $V_2 = V_1 \setminus \{ x_n,x_{n_{r+1}},\ldots,x_{n_s}  \}.$\\
Since $\displaystyle{ \sum_{x \in V_2 } } w(x)  = w_{n-1} + (w_{n_1} + w_{n_2}+ \cdots + w_{n_r} ) + (w_{{n-1}_1}+w_{{n-1}_2}+\cdots+w_{{n-1}_p}    ) +  (w_{{n-1}_{p+1}}+w_{{n-1}_{p+2}}+\cdots+w_{{n-1}_t}    )    \geq 2 + r + p + 2(t-p) = (t-p) + (t+r+2)    ,$ 
 $  \reg(J \cap   K) - 1 \leq   \reg(K).$  Thus by Lemma \ref{pol} and Corollary \ref{betti.2}, we have $$\reg(I(D^\prime_k))=\reg(I(D^\prime_k)^{\mathcal{P}}) = \max\{ \reg(J) , \reg(K) , \reg(J \cap K) - 1 \} = \reg(K) = \reg(I(D)).$$           
  Case-II: Let  $ x_{n-2} \notin N_D^+ (x_n) $. If   $ x_{n-2} $ is a leaf, then $x_{n-2}$ become a source vertex which implies   
  $w_{n-2} = 1$ and by the property $P$, $w_{n-1} \geq 1 .$  Then we follow the same  process  of Case-I and see that the value of $ \reg(J \cap K) $ remains same  as in Case-I  where only  $V_2$ is replaced by $V_2 \cup \{x_{n-2}\}$. If   $ x_{n-2} $ is not   a leaf,   
 $w_{n-2} \geq 1  $   and  by the property $P,$ $w_{n-1} \geq 2 .$ Again we follow the same  process  of Case-I and see that the value of $ \reg(J \cap K) $ remains same  as in Case-I  where $V_2$ also remains same.            
 Wheather $x_{n-2}$ is a leaf or non-leaf vertex,  $\displaystyle{ \sum_{x \in V_2 } } w(x)  \geq (t-p) + (t+r+2).$  Therefore by the similar arguement as in Case-I, $\reg(I(D^\prime_k))  =  \reg(I(D)).$    

 Case-B: Assume $   N_D^- (x_{n-1}) = \emptyset $.\\ By the property $P$ of $D,$ $x_{n-1}$ is a source vertex. This implies that $w_{n-1} = 1.$\\
Then by the similar arguement as in Case-I of Case-A, 
 \begin{align*}
 \reg(J\cap K)-1 &= \reg(J) + \reg(L) -1\\
 \hspace*{2.43cm} &= \reg(J) +  \reg(L_1 + L_2  )-1 \\
 \hspace*{2.43cm} &= (w_n + 1) + \displaystyle{\sum_{x \in V(D) \setminus V_1} }   w(x) + (1+w_{n_{r+1}}-1) + (1+w_{n_{r+2}}-1)\\
 &\hspace*{0.4cm}+ \cdots + (1+w_{n_{s}}-1) +(t-p) - [|E(D)| - (t+r+1)] + 1 -1 \\
 &= \displaystyle{ \sum_{x \in V(D) \setminus V_2}} w(x)   -|E(D)| + (t-p) + (t+r+1) + 1     
 \end{align*}  
 where $V_1 =  \{x_n,x_{n-1},x_{{n_1}},x_{{n_2}},\ldots,x_{{n_r}},x_{n_{r+1}},\ldots,x_{n_s},x_{{n-1}_1},\ldots,x_{{n-1}_p},x_{{n-1}_{p+1}}, \ldots,x_{{n-1}_{t}}\}$
 \\ and $V_2 = V_1 \setminus \{ x_n,x_{n_{r+1}},\ldots,x_{n_s}  \}.$\\
 Since $\displaystyle{ \sum_{x \in V_2 } } w(x)  = w_{n-1} + (w_{n_1} + w_{n_2}+ \cdots + w_{n_r} ) + (w_{{n-1}_1}+w_{{n-1}_2}+\cdots+w_{{n-1}_p}    ) +  (w_{{n-1}_{p+1}}+w_{{n-1}_{p+2}}+\cdots+w_{{n-1}_t}    )    \geq 1 + r + p + 2(t-p) = (t-p) + (t+r+1)    ,$ 
 $  \reg(J \cap   K) - 1 \leq   \reg(K).$  Thus  by the similar arguement as in Case-I of Case-A, $\reg(I(D^\prime_k))  =  \reg(I(D)).$

\end{proof}

\begin{proposition}\label{exteded-s.}
	Let $ D_1 , D_2 , \ldots, D_s $ for $s \geq 2$ are the  weighted   oriented graphs having
	 property $P$ with weight function $ w $  
	on vertex sets $ \{x_{1_1},\ldots,x_{{n_1}_1}\} ,$$ \{x_{1_2},\ldots,x_{{n_2}_2}\} , \ldots , \{x_{1_s},\\\ldots,x_{{n_s}_s}\} $  respectively. 
	 Let  $ D   $ be a weighted oriented graph obtained by adding $ k $ new oriented edges among $ D_1 , D_2 , \ldots, D_s $   where every edge is of the form $ (x_{a_i}
	, x_{b_j} ) $     for some $ x_{a_i} \in  V(D_i),$  $x_{b_j}  \in V(D_j) ,$  $i \neq j $ with  $ w_{a_i}, w_{b_j} \geq 2 $ such that  no vertex of  $ N^-_{D_j}(x_{b_j} ) $ is a leaf vertex in   $D_j$   and the set of new edges oriented towards $D_t$  go to a single vertex of $D_t$ for $t=1,\ldots,s. $ 
	Let $k_1,\ldots,k_s$ number of new edges are oriented towards      $D_1,\ldots,D_s$ where $k_1+\cdots+k_s=k$  and $D^\prime_{t}$ be the new weighted oriented graph  after addition of the  $k_t$ new oriented edges to $D_t$ which are oriented towards a single vertex of it for $t=1,\ldots,s. $              Then
	$$ \reg(R/I(D)) =\reg(R/I(D^\prime_{1})) +\cdots+ \reg(R/I(D^\prime_{s})).    $$              
\end{proposition}
\begin{proof}
By Theorem \ref{exteded-s}, we have $\reg(R/I(D)) =\reg(R/I(D_1)) +\cdots+ \reg(R/I(D_s))$ and by Theorem \ref{exteded-k}, $\reg(R/I(D_t)) =\reg(R/I(D^\prime_{t}))  $ for   $t=1,\ldots,s. $ Hence   $ \reg(R/I(D)) =\reg(R/I(D^\prime_{1})) +\cdots+ \reg(R/I(D^\prime_{s})).   $  
\end{proof}
The above proposition partially answer  the following question asked by H.T. H{\`a} in \cite{htha}.
\begin{question}\cite[Problem 6.8]{htha}
 Let $ \mathcal{H}, \mathcal{H}_1,\ldots,\mathcal{H}_s $ be simple hypergraphs over the same vertex set $ X $ and
assume that $ \displaystyle{ \mathscr{E}(\mathcal{H}) = \bigcup_{i=1}^{s}
	\mathscr{E}(\mathcal{H}_i). }$ Find combinatorial conditions for the following equality
to hold:
$$ \reg(S/I(\mathcal{H})) = \sum_{i=1}^{s}\reg(S/I(\mathcal{H}_i)).$$

\end{question}
\textbf{Observation:}
Let $ D, D^\prime_{1}, D^\prime_{2}, \ldots, D^\prime_{s} $ and $R$  are same as defined in Proposition \ref{exteded-s.}. Let $X =  \{x_{1_11},\ldots,x_{1_1w_{1_1}},\ldots,x_{{n_1}_11},\ldots,x_{{n_1}_1w_{{n_1}_1}},  \ldots ,x_{1_s1},\ldots,x_{1_sw_{1_s}},\ldots,x_{{n_s}_s1},\ldots,x_{{n_s}_sw_{{n_s}_s}}\}$ and $S = R^{\mathcal{P}}.$ If we assume that $ \mathcal{H}, \mathcal{H}_1,\ldots,\mathcal{H}_s $ be the simple hypergraphs  over  $X$ such that  $I(D)^\mathcal{P},I(D^\prime_{1})^\mathcal{P},\ldots,I(D^\prime_{s})^\mathcal{P}$  are the square-free monomial  edge ideals   $ I(\mathcal{H}), I(\mathcal{H}_1),  \ldots,\\I(\mathcal{H}_s) $ respectively then   $ \displaystyle{ \mathscr{E}(\mathcal{H}) = \bigcup_{i=1}^{s}
	\mathscr{E}(\mathcal{H}_i) } $ and  by Proposition \ref{exteded-s.} and Lemma \ref{pol},\\  $\displaystyle{ \reg(S/I(\mathcal{H})) = \sum_{i=1}^{s}\reg(S/I(\mathcal{H}_i))}.$

\section{Regularity in weighted Oriented Paths and Cycles}
 In this section for a weighted
 oriented path or cycle, we relate the regularity of its edge ideal to the regularity
 of edge ideal of its underlying graph when vertices of  $ V^+ $ are sinks. Firstly, we compute the regularity of edge ideals of weighted oriented paths when  vertices of $ V^+ $ are sinks. \\
We divide the set $ T $ of all weighted oriented paths when
  vertices of $ V^+ $ are sinks into two sets:\\
 $ T_1 $: Set of all weighted oriented paths where the two end vertices are in $ V^+ $ and the distance between any two consecutive  vertices of $ V^+ $ is 3.\\
 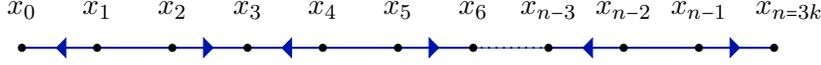
\begin{figure}[!ht]\label{fig.4}
 	\begin{tikzpicture}
 	\begin{scope}[ thick, every node/.style={sloped,allow upside down}] 
 	\definecolor{ultramarine}{rgb}{0.07, 0.04, 0.56} 
 	\definecolor{zaffre}{rgb}{0.0, 0.08, 0.66}   
 	\draw[fill, zaffre] (1,0) --node {\midarrow}(0,0);
 	\draw[fill, zaffre][-,thick] (1,0) --(2,0);
 	\draw[fill, zaffre] (2,0) --node {\midarrow}(3,0);
 	\draw[fill, zaffre] (4,0) --node {\midarrow}(3,0);
 	\draw[fill, zaffre][-,thick] (4,0) --(5,0);
 	\draw[fill, zaffre] (5,0) --node {\midarrow}(6,0);
 	\draw[fill, zaffre][dotted,thick] (6,0) --(7,0);
 	\draw [fill, zaffre](8,0) --node {\midarrow}(7,0);
 	\draw[fill, zaffre][-,thick] (8,0) --(9,0);
 	\draw[fill, zaffre] (9,0) --node {\midarrow}(10,0);
 	\draw [fill] [fill] (0,0) circle [radius=0.04];
 	\draw[fill] [fill] (1,0) circle [radius=0.04];
 	\draw[fill] [fill] (2,0) circle [radius=0.04];
 	\draw[fill] [fill] (3,0) circle [radius=0.04];
 	\draw[fill] [fill] (4,0) circle [radius=0.04];
 	\draw[fill] [fill] (5,0) circle [radius=0.04];
 	\draw[fill] [fill] (6,0) circle [radius=0.04];
 	\draw[fill] [fill] (7,0) circle [radius=0.04];
 	\draw[fill] [fill] (8,0) circle [radius=0.04];
 	\draw[fill] [fill] (9,0) circle [radius=0.04];
 	\draw[fill] [fill] (10,0) circle [radius=0.04];
 	\node at (0,0.5) {$x_0$};
 	\node at (1,0.5) {$x_1$};
 	\node at (2,0.5) {$x_2$};
 	\node at (3,0.5) {$x_{3}$};
 	\node at (4,0.5) {$x_4$};
 	\node at (5,0.5) {$x_5$};
 	\node at (6,0.5) {$x_6$};
 	\node at (7,0.5) {$x_{n-3}$};
 	\node at (8,0.5) {$x_{n-2}$};
 	\node at (9,0.5) {$x_{n-1}$};
 	\node at (10.2,0.5) {$x_{n=3k}$};   
 	\end{scope}
 	\end{tikzpicture}
 	\caption{A weighted oriented path in $ T_1 $.}
 \end{figure}\\
 Note that the length of any weighted oriented path in $T_1$ is a multiple of $3.$ (See Figure \ref{fig.4}.)\\
$ T_2 $: Set of remaining weighted oriented paths when the vertices of $ V^+ $ are sinks, i.e.,  \hspace*{0.8cm}$   T_2 = T \setminus T_1. $ 
\begin{remark} \label{T} 
	Let $D$   be a weighted oriented path of length $n$ for $ n \geq 4 $ in $ T_2 $  with underlying graph $ G= P_n=x_0 x_1 \ldots x_n .$ Let $ D_1 = D\setminus \{x_n\} $, $ D_2 = D\setminus \{ x_{n-2},x_{n-1},x_n \} ,D^\prime_1 = D\setminus \{x_0\} $ and $ D^\prime_2 = D\setminus \{ x_{0},x_{1},x_2 \}.$\\
  Case-I: Assume $ n\equiv 1~(\mbox{mod}~3).$ Here $n\geq 4$ and $n=3k+1$ for some $k\in \mathbb{N}.$ Then the length of $D_2$ or $D^\prime_2$ is $3k-2$ which is not a multiple of $3.$ Thus both $D_2$ and $D^\prime_2$ are in $T_2.$ If  $D_1$ is in $T_1,$ then $x_{n-1}  \in V^+$ which implies  $x_{n}  \notin V^+.$ So  one end vertex of $D^\prime_1,$ i.e., $x_{n}  \notin V^+.$  Hence $D^\prime_1$  is in $T_2.$\\
	Case-II: Assume $ n\equiv 2~(\mbox{mod}~3).$ Here $n \geq 5$ and $n=3k+2$ for some $k\in \mathbb{N}.$ Then the length of $D_1$ and $D_2$ are $3k+1$ and $3k-1$ respectively. Note that none of them is  a multiple of $3.$ Hence both $D_1$ and $D_2$ are in $T_2.$\\
	Case-III: Assume $ n\equiv 0~(\mbox{mod}~3).$ Here $n\geq 6$ and $n=3k$ for some $k\in \mathbb{N}.$  Then the length of $D_1$ or $D^\prime_1$ is $3k-1$ which is not a multiple of $3.$ Thus both $D_1$ and $D^\prime_1$ are in $T_2.$ If  $D_2$ is in $T_1,$ then $x_{n-3}  \in V^+.$ Since $D$ is in $T_2,$  $x_{n}  \notin V^+$ which implies one end vertex of $D^\prime_2,$ i.e., $x_{n}  \notin V^+.$  Hence $D^\prime_2$  is in $T_2.$\\ 
	Observe that  if $D$ $\in  $ $T_2$ then $D_1,$  $D_2$ $\in  $ $T_2$ or $D^\prime_1,$  $D^\prime_2$ $\in  $ $T_2$. Thus without loss of generality we can rename the vertices and always assume that $D_1$ and $D_2$ are in $T_2.$            
\end{remark}

\begin{theorem} \label{ClassII}
 Let $ D $ be a weighted oriented path of length $ n $ in $ T_2 $  with underlying graph $ G= P_n=x_0 x_1 \ldots x_n .$ Then $\reg(I(D)) = \reg(I(G)) + \displaystyle{ \sum_{x_i \in V^+}}(w_i - 1) $ where $ w_i=w(x_i)  $ for $x_i \in V^+.$ 
\end{theorem}

\begin{proof}
	Here $ V(D)=\{x_0,x_1,\ldots,x_n\} $. 
	 We use the method of induction on the number of edges of $D$  and prove this theorem in different cases depending upon the position of the sink vertices.\\
	Base Case: $ |E(D)| \leq  3.$\\ Assume that $ |E(D)| =  3$ and $ V(D) = \{x_0,x_1,x_2,x_3\}.$
	 If $x_0,x_1 \notin V^+ $ and $x_3 \in V^+,$ then
	$I(D)=(x_0x_1,x_1x_2,x_2x_3^{w_3}).$
	 Let $ J=(x_2x_3^{w_3}) $ and $ K=(x_0x_1,x_1x_2) .$ Since $ J $ has linear resolution, $ I(D)=J+K $ is a Betti splitting. Here $ \reg(J)=w_3 + 1 $ and $ \reg(K)= 2. $ Note that $J\cap K= JL  $ where $ L = (x_1) .$ This implies that $ \reg(J\cap K) = w_3 + 2  .$ Thus by Corollary \ref{betti.2}, we have $\reg(I(D)) = \max\{ \reg(J),\reg(K), \reg(J\cap K) - 1\}$  $= w_3+ 1$  $= \reg(I(G)) +  w_3 - 1.$ Similarly depending upon the position of sink vertices using the Betti splitting technique for any weighted oriented path $D$ in $T_2$ with $ |E(D)| \leq 3 ,$ we can show that  $\reg(I(D)) =  \reg(I(G)) + \displaystyle{ \sum_{x_i \in V^+}}(w_i - 1).$\\    
	Now we consider $D$  to be a weighted oriented path of length $ n \geq 4 $ and $ V(D) = \{x_0,\ldots,x_n\}.$ Let $ D_1 = D\setminus \{x_n\} $, $ D_2 = D\setminus \{ x_{n-2},x_{n-1},x_n \} ,$ $ G_1 = G\setminus \{x_n\} $ and $ G_2 = G\setminus \{ x_{n-2},x_{n-1},x_n \} ,$ i.e., $ G_1 $ and $ G_2 $ are the corresponding underlying graphs of $D_1 $ and  $ D_2 $ respectively. Without loss of generality by Remark \ref{T}, we can fix $x_n$ in one end of $D$ such that $ D_1 $ and $ D_2 $ are in $ T_2 $.\\ 
	Case-I: Assume that $ x_{n-2} \notin V^+  $  and  $ x_{n} \in V^+  $. 
	Let $ J = (x_{n-1}x_n^{w_n}) $ and $ K=I(D_1). $ As $ J $ has linear resolution, $ I(D) = J + K $ is a Betti splitting and $ \reg(J)= w_n + 1 $. Since $ D_1 $ is the weighted oriented path of length $ n-1 $ in $ T_2 $, by induction hypothesis we get $ \reg(K)=\reg(I(D_1))= \reg(I(G_1)) + \displaystyle{ \sum_{x_i \in V^+ \setminus \{x_n\} }}(w_i - 1)  .$ Note that $ J \cap K= JL $ where $ L= (I(D_2),x_{n-2}) .$ Since $ D_2 $ is the weighted oriented path of length $ n-3 $ in $ T_2 $, by induction hypothesis we have $\reg(L)=\reg(I(D_2))= \reg(I(G_2)) + \displaystyle{ \sum_{x_i \in V^+ \setminus \{x_n\} }}(w_i - 1)   .$  By Corollary \ref{path}, $ \reg(I(G))= \reg(I(G_2)) + 1 .$ Thus by Lemma \ref{reg.3}, we have   
	 \begin{align*}
	 \reg(J\cap K) &= \reg(J) + \reg(L)\\
	 &=\reg(J) + \reg(I(D_2))\\&=(w_n + 1) + \reg(I(G_2)) + \displaystyle{ \sum_{x_i \in V^+ \setminus \{x_n\} }}(w_i - 1)\\&= \reg(I(G_2)) + 1 + \displaystyle{ \sum_{x_i \in V^+ \setminus \{x_n\} }}(w_i - 1) + w_n\\&= \reg(I(G)) + \displaystyle{ \sum_{x_i \in V^+ \setminus \{x_n\} }}(w_i - 1) + w_n
	 \end{align*} 
By Lemma \ref{H<G}, $\reg(I(G_1)) \leq \reg(I(G))	.$  Thus by Corollary \ref{betti.2}, we get    
	 \begin{align*}
         \reg(I(D)) &= \max\{ \reg(J),~\reg(K), ~\reg(J\cap K) - 1\}\\&= \max\{ w_n + 1,~\reg(I(G_1)) + \displaystyle{ \sum_{x_i \in V^+ \setminus \{x_n\} }}(w_i - 1) ,~ \reg(I(G)) + \displaystyle{ \sum_{x_i \in V^+}}(w_i - 1)  \} \\&=\reg(I(G)) + \displaystyle{ \sum_{x_i \in V^+}}(w_i - 1). 
	 \end{align*}
Case-II: Assume that  $ x_{n-2}, x_{n} \in V^+  $.\\ 
Let $ J = (x_{n-1}x_n^{w_n}) $ and $ K=I(D_1). $ Since $ J $ has linear resolution, $ I(D) = J + K $ is a Betti splitting. Here $ \reg(J)= w_n + 1 $ and by the same arguement as in Case-I, 
$\reg(K)= \reg(I(G_1)) + \displaystyle{ \sum_{x_i \in V^+ \setminus \{x_n\}  }}(w_i - 1)  $ and $\reg(J \cap K ) =  \reg(I(G)) + \displaystyle{ \sum_{x_i \in V^+ \setminus \{x_n\} }}(w_i - 1) + w_n.$ By the same arguement as in Case-I and Corollary \ref{betti.2}, $\reg(I(D))= \reg(I(G)) + \displaystyle{ \sum_{x_i \in V^+  }}(w_i - 1)   .  $ \\ 
Case-III: Assume that  $ x_{n-1} \in V^+  $.\\ 
Let $ J = (x_{n}x_{n-1}^{w_{n-1}}) $ and $ K=I(D_1). $ Since $ J $ has linear resolution, $ I(D) = J + K $ is a Betti splitting. Here $ \reg(J)= w_{n-1} + 1  $ and 
by the same arguement as in Case-I, 
$\reg(K)= \reg(I(G_1)) + \displaystyle{ \sum_{x_i \in V^+   }}(w_i - 1)  $ and $\reg(J \cap K ) =  \reg(I(G)) + \displaystyle{ \sum_{x_i \in V^+ \setminus \{x_{n-1}\} }}(w_i - 1) + w_{n-1}.$ By the same arguement as in Case-I and Corollary \ref{betti.2}, $\reg(I(D))= \reg(I(G)) + \displaystyle{ \sum_{x_i \in V^+  }}(w_i - 1)   .  $\\
Case-IV: Assume that $ x_{n-2} \in V^+ , x_{n-1} \notin V^+  $  and  $ x_{n} \notin V^+  $.\\ 
Let $ J = (x_{n-1}x_n) $ and $ K=I(D_1)  .$  Since $ J $ has linear resolution, $ I(D) = J + K $ is a Betti splitting. Here $ \reg(J)= 2 $   and by the same arguement as in Case-I, 
$\reg(K)= \reg(I(G_1)) + \displaystyle{ \sum_{x_i \in V^+   }}(w_i - 1)  $ and $\reg(J \cap K ) =  \reg(I(G)) + \displaystyle{ \sum_{x_i \in V^+ \setminus \{x_{n-2}\} }}(w_i - 1) + w_{n-2}.$ By the same arguement as in Case-I and Corollary \ref{betti.2}, $\reg(I(D))= \reg(I(G)) + \displaystyle{ \sum_{x_i \in V^+  }}(w_i - 1)   .  $\\
Case-V: Assume that $ x_{n-2},x_{n-1},x_{n} \notin V^+  .$ \\ 
Let $ J = (x_{n-1}x_n) $ and $ K=I(D_1). $ Since $ J $ has linear resolution, $ I(D) = J + K $ is a Betti splitting. Here $ \reg(J)= 2$ and  by the same arguement as in Case-I, 
$\reg(K)= \reg(I(G_1)) + \displaystyle{ \sum_{x_i \in V^+   }}(w_i - 1)  $ and $\reg(J \cap K ) =  \reg(I(G)) + \displaystyle{ \sum_{x_i \in V^+ }}(w_i - 1) + 1 .$ By the same arguement as in Case-I and Corollary \ref{betti.2}, $\reg(I(D))= \reg(I(G)) + \displaystyle{ \sum_{x_i \in V^+  }}(w_i - 1)   .  $\\
 Hence for any  weighted oriented path $D$ of length $ n $ in $ T_2 $,  $$\reg(I(D)) = \reg(I(G)) + \displaystyle{ \sum_{x_i \in V^+}}(w_i - 1)  ~\mbox{where}~  w_i=w(x_i)   ~\mbox{for}~ x_i \in V^+.$$ 
	 
\end{proof}
 
 \begin{theorem} \label{ClassI}
 	Let $ D $ be a weighted oriented path of length $ n $ in $ T_1 $ with  underlying graph $ G= P_n=x_0 x_1 \cdots x_n .$ Then $\reg(I(D)) = \reg(I(G)) + \displaystyle{ \sum_{x_i \in V^+\setminus\{x_j\}}}(w_i - 1) $   where $ w_i=w(x_i)  $ for $x_i \in V^+$ and $ x_j $ is one of  the vertices of $ V^+ $ with minimum weight. 
 \end{theorem}
  \begin{proof}
  Here $ V(D)=\{x_0,x_1,\ldots,x_n\} $. 	
   By the definition of $ T_1 ,$  $ G=P_n=P_{3k} $	for some $ k\in \mathbb{N} .$ We use the method of induction on $ k. $ 
  	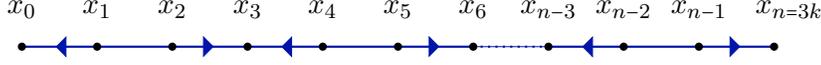
\begin{figure}[!ht]\label{fig.5}
  	\begin{tikzpicture}
  	\begin{scope}[ thick, every node/.style={sloped,allow upside down}] 
  	\definecolor{ultramarine}{rgb}{0.07, 0.04, 0.56} 
  		\definecolor{zaffre}{rgb}{0.0, 0.08, 0.66}   
  	\draw[fill, zaffre] (1,0) --node {\midarrow}(0,0);
  	\draw[fill, zaffre][-,thick] (1,0) --(2,0);
  	\draw[fill, zaffre] (2,0) --node {\midarrow}(3,0);
  	\draw[fill, zaffre] (4,0) --node {\midarrow}(3,0);
  	\draw[fill, zaffre][-,thick] (4,0) --(5,0);
  	\draw[fill, zaffre] (5,0) --node {\midarrow}(6,0);
  	\draw[fill, zaffre][dotted,thick] (6,0) --(7,0);
  	\draw [fill, zaffre](8,0) --node {\midarrow}(7,0);
  	\draw[fill, zaffre][-,thick] (8,0) --(9,0);
  	\draw[fill, zaffre] (9,0) --node {\midarrow}(10,0);
  	\draw [fill] [fill] (0,0) circle [radius=0.04];
  	\draw[fill] [fill] (1,0) circle [radius=0.04];
  	\draw[fill] [fill] (2,0) circle [radius=0.04];
  	\draw[fill] [fill] (3,0) circle [radius=0.04];
  	\draw[fill] [fill] (4,0) circle [radius=0.04];
  	\draw[fill] [fill] (5,0) circle [radius=0.04];
  	\draw[fill] [fill] (6,0) circle [radius=0.04];
  	\draw[fill] [fill] (7,0) circle [radius=0.04];
       \draw[fill] [fill] (8,0) circle [radius=0.04];
  	\draw[fill] [fill] (9,0) circle [radius=0.04];
  	\draw[fill] [fill] (10,0) circle [radius=0.04];
  	\node at (0,0.5) {$x_0$};
  	\node at (1,0.5) {$x_1$};
  	\node at (2,0.5) {$x_2$};
  	\node at (3,0.5) {$x_{3}$};
  	\node at (4,0.5) {$x_4$};
  	\node at (5,0.5) {$x_5$};
  	\node at (6,0.5) {$x_6$};
  	\node at (7,0.5) {$x_{n-3}$};
  	\node at (8,0.5) {$x_{n-2}$};
  	\node at (9,0.5) {$x_{n-1}$};
  	\node at (10.2,0.5) {$x_{n=3k}$};   
  	\end{scope}
  	\end{tikzpicture}
  	\caption{A weighted oriented path in $ T_1 $.}
  \end{figure}\\
  Base Case: If $ k=1 ,$ then $ I(D)=(x_0^{w_0}x_1,x_1x_2,x_2x_3^{w_3}).$ Let $ J=(x_2x_3^{w_3}) $ and $ K=(x_0^{w_0}x_1,x_1x_2) .$ Since $ J $ has linear resolution, $ I(D)=J+K $ is a Betti splitting. Here $ \reg(J)=w_3 + 1 $ and $ \reg(K)=w_0 + 1. $ Note that $J\cap K= JL  $ where $ L = (x_1) .$ This implies that $ \reg(J\cap K) = w_3 + 2  .$ Thus by Corollary \ref{betti.2}, we have $\reg(I(D)) = \max\{ \reg(J),\reg(K),\\ \reg(J\cap K) - 1\}$  $= \max\{ w_0 + 1,w_3+ 1\}$  $= 2 +  \max\{ w_0-1 ,w_3 - 1\}$  $= \reg(I(G)) +  \max\{ w_0-1 ,w_3 - 1\}.$ \\
  Now we consider the case $ n=3k $ for some $ k > 1.  $  Let $ D_1 = D\setminus \{x_n\} $, $ D_2 = D\setminus \{ x_{n-2},x_{n-1},x_n \} ,$ $ G_1 = G\setminus \{x_n\} $ and $ G_2 = G\setminus \{ x_{n-2},x_{n-1},x_n \} ,$ i.e., $ G_1 $ and $ G_2 $ are the corresponding underlying graphs of $D_1 $ and  $ D_2 $ respectively.  Here $ I(D)=(x_0^{w_0}x_1,x_1x_2,\\x_2x_3^{w_3},\ldots,x_{n-3}^{w_{n-3}}x_{n-2},x_{n-2}x_{n-1},x_{n-1}x_n^{w_n}).$  Let $ J = (x_{n-1}x_n^{w_n}) $ and $ K=I(D_1). $ Since $ J $ has linear resolution, $ I(D) = J + K $ is a Betti splitting. Here $ \reg(J)= w_n + 1 .$  Since $ x_{n-1} \notin   V^+ $, $ D_1 $ is a weighted oriented path of length $ n-1 $ in $ T_2 .$ By  Corollary \ref{path}, $ \reg(I(G))= \reg(I(G_1)) .$ Thus by Theorem \ref{ClassII}, we have  $$ \reg(K)=\reg(I(D_1))= \reg(I(G_1)) + \displaystyle{ \sum_{x_i \in V^+ \setminus \{x_n\} }}(w_i - 1) = \reg(I(G)) + \displaystyle{ \sum_{x_i \in V^+ \setminus \{x_n\} }}(w_i - 1)  .$$ Note that $ J \cap K= JL $ where $ L= (I(D_2),x_{n-2}) $ and $ D_2 $ is a weighted oriented path of length $n-3=3(k-1)$ in $ T_1. $ Thus   by the induction hypothesis, we get   $$ \reg(L)=\reg(I(D_2))= \reg(I(G_2)) + \displaystyle{ \sum_{x_i \in V^+ \setminus \{ x_n,x_m \} }}(w_i - 1)  $$where $ x_m $  is one of vertices of $    V^+(D_2) $   with minimum weight.    By Corollary \ref{path},  $ \reg(I(G))  = \reg(I(G_2)) + 1 .$ Thus by Lemma \ref{reg.3}, we have    
  \begin{align*}
  \reg(J\cap K) &= \reg(J) + \reg(L)\\
  &=\reg(J)   + \reg(I(D_2))\\&=(w_n + 1) + \reg(I(G_2)) + \displaystyle{ \sum_{x_i \in V^+ \setminus \{ x_n,x_m \} }}(w_i - 1)\\&= \reg(I(G_2)) + 1 + \displaystyle{ \sum_{x_i \in V^+ \setminus \{ x_n,x_m \} }}(w_i - 1) + w_n\\&= \reg(I(G)) + \displaystyle{ \sum_{x_i \in V^+ \setminus \{ x_n,x_m \} }}(w_i - 1) + w_n
  \end{align*} 
    Therefore by Corollary \ref{betti.2}, we get 
  \begin{align*}
  \reg(I(D)) &= \max\{ \reg(J),~\reg(K),~ \reg(J\cap K) - 1\}\\&= \max\{ w_n + 1,~\reg(I(G)) + \displaystyle{ \sum_{x_i \in V^+ \setminus \{x_n \}}}(w_i - 1) ,~ \reg(I(G)) + \displaystyle{ \sum_{x_i \in V^+ \setminus \{ x_m \} }}(w_i - 1)  \} \\&=\reg(I(G)) + \displaystyle{ \sum_{x_i \in V^+  \setminus \{ x_j \}}}(w_i - 1)
  \end{align*}
  where $ x_j=\min\{x_n,x_m\},$ i.e., $ x_j $ is one of the  vertices of $ V^+ $ with minimum weight.\\
  Hence for any weighted oriented path $D$ of length $ n  $ in $ T_1 $, 
  $$\reg(I(D)) = \reg(I(G)) + \displaystyle{ \sum_{x_i \in V^+\setminus\{x_j\}}}(w_i - 1)$$  where $ w_i=w(x_i)  $ for $x_i \in V^+$ and $ x_j $ is one of  the vertices of $ V^+ $ with minimum weight.	   
  \end{proof}
\begin{theorem} \label{cycle1}
	Let $ D $ be a weighted oriented cycle of length $ n $ for $ n\equiv 0,1~(\mbox{mod}~3)$  with  underlying graph  $ G= C_n=x_1\ldots x_nx_1$    and  vertices of $ V^+ $ are sinks. Then $\reg(I(D)) = \reg(I(G)) + \displaystyle{ \sum_{x_i \in V^+}}(w_i - 1) $ where $ w_i=w(x_i)  $ for $x_i \in V^+.$   
\end{theorem}

\begin{proof}
Here $ V(D)=\{x_1,\ldots,x_n\} $. 
	Without loss of generality, let  $ x_k\neq x_1,x_n$ be one of the vertices of $ V^+ $.  
	  Let $ D_1 = D \setminus \{x_k\} ,$ $ D_2 = D \setminus N_D[x_k] ,$ $ G_1=G \setminus \{x_k\} $ and  $ G_2 = G \setminus N_G[x_k] ,$ i.e., $ G_1 $ and $ G_2 $ are the corresponding underlying graphs of $D_1 $ and  $ D_2 $ respectively.  
        By Corollary \ref{cycle}, $\reg(I(G))   =  \reg(I(G_1))=\reg(I(G_2)) + 1    $ except $n=3,4$ 
	and $\reg(I(G))   =  \reg(I(G_1))=2$ for $n=3,4$.\\       
Now consider the exact sequence
	\begin{equation}\label{2}
	0 \longrightarrow {\frac{R}{(I(D):x_k^{w_k}) }}(-w_k) \xrightarrow{.x_k^{w_k}} {\frac{R}{I(D) }} \longrightarrow {\frac{R}{(I(D),x_k^{w_k}) }} \longrightarrow 0
	\end{equation}		
	Here $(I(D),x_k^{w_k })= (I(D_1),x_k^{w_k })$ where $ D_1 = D \setminus \{x_k\} $ is a weighted oriented path of length $ n-2 $ with end vertices $x_{k-1}$ and $x_{k+1}$.  Since $  x_k \in   V^+  $ and vertices of $   V^+ $ are sinks,  $x_{k-1},x_{k+1} \notin  V^+.$ Hence $ D_1  \in  T_2 $. Thus by  Lemma \ref{reg.2} and Theorem \ref{ClassII}, we have
	 \begin{align*}
\reg((I(D),x_k^{w_k}))&=\reg(I(D_1))+ w_k - 1\\
&=\reg(I(G_1))   +  \displaystyle{ \sum_{x_i \in V^+ \setminus \{x_k\}}}(w_i - 1) + w_k -1 \\
\hspace*{6.2cm} &= \reg(I(G_1)) + \displaystyle{ \sum_{x_i \in V^+ }}(w_i - 1).
    \end{align*}
Here $(I(D):x_k^{w_k })	=(I(D_2),x_{k-1},x_{k+1})$ except $n=3,4$\\
\hspace*{0.08cm} and 
$(I(D):x_k^{w_k })=
(x_{k-1},x_{k+1})$ for $n=3,4.$\\	    
For $n \neq 3,4,$ since $ D_2  $ is a weighted oriented path of length $ n-4 $   in $ T_1 $  or  $ T_2 ,$ \\ by Theorem    \ref{ClassII} and Theorem \ref{ClassI}, we have
\begin{align*}
 \reg((I(D):x_k^{w_k })(-(w_k)))&=\reg(I(D_2)) + w_k\\
 &\leq \reg(I(G_2)) + \displaystyle{ \sum_{x_i \in V^+ \setminus \{x_k\} } }(w_i - 1) + w_k\\
 &=\reg(I(G_2)) + \displaystyle{ \sum_{x_i \in V^+ }}(w_i - 1) + 1\\
 &=\reg(I(G_1)) + \displaystyle{ \sum_{x_i \in V^+ }}(w_i - 1).
\end{align*}

 For $n=3,4,$ $\reg((I(D):x_k^{w_k  })(-w_k))= 1 + w_k\\
	\hspace*{6.1cm}=\reg(I(G_1)) + w_k - 1 \\
	\hspace*{6.1cm} \leq   \reg(I(G_1)) + \displaystyle{ \sum_{x_i \in V^+ }}(w_i - 1).$ 

  By Lemma \ref{reg.4} and exact sequence (\ref{2}), we get
	\hspace*{4cm} $$ \reg(I(D)) \leq \max\{\reg((I(D):x_k^{w_k  })(-w_k)),\reg((I(D),x_k^{w_k }))\} .$$
	Since $\reg((I(D):x_k^{w_k  })(-w_k)) - 1 \neq \reg((I(D),x_k^{w_k })),$ by Lemma  \ref{reg.4} and exact \\sequence (\ref{2})  we have
	\begin{align*}
   \reg(I(D)) &=\reg((I(D),x_k^{w_k}))\\
     &= \reg(I(G_1)) + \displaystyle{ \sum_{x_i \in V^+ }}(w_i - 1)\\
     &= \reg(I(G)) + \displaystyle{ \sum_{x_i \in V^+ }}(w_i - 1).
    \end{align*}

\end{proof}

\begin{theorem} \label{cycle2}
	Let $ D $ be a weighted oriented cycle of length $ n $ for $ n\equiv 2~(\mbox{mod}~3)$  with  underlying graph  $ G= C_n=x_1\ldots x_nx_1$    and  vertices of $ V^+ $ are sinks. Then $\reg(I(D)) = \reg(I(G)) + \displaystyle{ \sum_{x_i \in V^+}}(w_i - 1) $ where $ w_i=w(x_i)  $ for $x_i \in V^+.$ 
\end{theorem}

\begin{proof}
Here $ V(D)=\{x_1,\ldots,x_n\} $. We use the method of induction on the number of  vertices of   $  V^+. $\\   
 Base Case: Assume that $   V^+$ contains no  vertex. Then the proof follows trivially.\\
 Now we consider the case when $   V^+ $  contains $ m $ number of  vertices for some $ m\geq 1. $  Without loss of generality, let  $ x_k \neq x_1,x_n $ be one of the $ m $  vertices of $ V^+ .$ 
Let $ D_1 = D \setminus \{x_k\} $ and  $ G_1=G \setminus \{x_k\} ,$      i.e., $ G_1 $ is  the corresponding underlying graph of $D_1 $.      
By Corollary \ref{cycle}, we have $ \reg(I(G)) = \reg(I(G_1)) + 1.$ 	
	Consider the exact sequence  
	\begin{equation}\label{3}
	0 \longrightarrow {\frac{R}{(I(D):x_k^{w_k - 1}) }}(-(w_k-1)) \xrightarrow{.x_k^{w_k - 1}} {\frac{R}{I(D) }} \longrightarrow {\frac{R}{(I(D),x_k^{w_k - 1}) }} \longrightarrow 0
	\end{equation}
	Here $(I(D),x_k^{w_k - 1})= (I(D_1),x_k^{w_k - 1})$ where  $ D_1 = D \setminus \{x_k\} $ is a weighted oriented path of length $ n-2 $ with end vertices $x_{k-1}$ and $x_{k+1}$.  Since $  x_k \in   V^+  $ and vertices of $   V^+ $ are sinks, $x_{k-1},x_{k+1} \notin  V^+.$ Hence $ D_1  \in  T_2 $. Then by Theorem \ref{ClassII}, we have $   \reg(I(D_1))=  \reg(I(G_1)) +  \displaystyle{ \sum_{x_i \in V^+ \setminus \{x_k\}}}(w_i - 1) .$
	 Thus by  Lemma \ref{reg.2}, we have  
	 \begin{align*}
\reg((I(D),x_k^{w_k - 1}))&= \reg(I(D_1))+ (w_k -1) -1\\
&=  \reg(I(G_1))   +  \displaystyle{ \sum_{x_i \in V^+ }}(w_i - 1) - 1\\  
	&=  \reg(I(G)) +  \displaystyle{ \sum_{x_i \in V^+ }}(w_i - 1) - 2.
 \end{align*}

 Here $(I(D):x_k^{w_k - 1})=I(D_3)$ where $ D_3 $ is a weighted oriented cycle with      $m-1$  vertices in $V^+(D_3)$. Thus by using the induction hypothesis,$$ \reg(I(D_3))=  \reg(I(G)) +  \displaystyle{ \sum_{x_i \in V^+ \setminus \{x_k\}}}(w_i - 1) .$$Then   $\reg((I(D):x_k^{w_k - 1})(-(w_k-1)))=\reg(I(D_3)) + w_k - 1 =\reg(I(G)) +  \displaystyle{ \sum_{x_i \in V^+ }}(w_i - 1).$ By Lemma \ref{reg.4} and exact sequence (\ref{3}), we have
	$$ \reg(I(D)) \leq \max\{\reg((I(D):x_k^{w_k - 1})(-(w_k-1))),\reg((I(D),x_k^{w_k - 1}))\} .$$  
 Since $\reg((I(D):x_k^{w_k - 1})(-(w_k-1))) - 1  \neq  \reg((I(D),x_k^{w_k - 1})),$ by Lemma \ref{reg.4} and exact sequence (\ref{3}), we get 
 $$ \reg(I(D))= \reg((I(D):x_k^{w_k - 1})(-(w_k-1)))=\reg(I(G)) +  \displaystyle{ \sum_{x_i \in V^+ }}(w_i - 1) . $$  	Hence  for any weighted oriented cycle $ D $ of length $ n $ where $ n\equiv 2~(\mbox{mod}~3)$  with  vertices of $ V^+ $ are sinks,  $\reg(I(D)) = \reg(I(G)) + \displaystyle{ \sum_{x_i \in V^+}}(w_i - 1) $ where $ w_i=w(x_i)  $ for $x_i \in V^+.$ 
		
\end{proof}
By the computations in Macaulay 2, we have seen that it is even a hard job to relate the     regularity  of the edge ideal of a weighted oriented tree with the regularity of edge ideal of its underlying graph when the vertices of $ V^+ $  are sink.

\end{document}